\documentclass{amsart}

\usepackage{amsmath}
\usepackage{amssymb}
\usepackage{mathtools}
\usepackage{url}
\usepackage{hyperref}
\usepackage{autonum}
\usepackage{xspace}
\usepackage{xcolor}
\usepackage{hyperref}
\usepackage[msc-links]{amsrefs}

\usepackage{graphicx}
\usepackage{tikz} 
\usepackage{pgfplots} 

\usetikzlibrary{calc}
\usetikzlibrary{shapes.misc}

\tikzset{cross/.style={cross out, thick, draw=black, minimum size=2*(#1-\pgflinewidth), inner sep=0pt, outer sep=0pt},
cross/.default={2.5pt}}

\usepackage{caption}
\usepackage{subcaption}

\usepackage{siunitx}
\usepackage{multirow}

\newcommand{\Zeile}[5]{#1 & #2 & #3 & #4 & #5 }

\numberwithin{equation}{section}

\hypersetup{
		colorlinks,
		linkcolor={red!50!black},
		citecolor={blue!50!black},
		urlcolor={blue!80!black}
	}

\definecolor{color0}{rgb}{0.12156862745098,0.466666666666667,0.705882352941177}
\definecolor{color1}{rgb}{1,0.498039215686275,0.0549019607843137}
\definecolor{color2}{rgb}{0.172549019607843,0.627450980392157,0.172549019607843}
\definecolor{color3}{rgb}{0.83921568627451,0.152941176470588,0.156862745098039}
\definecolor{color4}{rgb}{0.580392156862745,0.403921568627451,0.741176470588235}
\definecolor{color5}{rgb}{0.549019607843137,0.337254901960784,0.294117647058824}
\definecolor{color6}{rgb}{0.890196078431372,0.466666666666667,0.76078431372549}

\definecolor{legendcolor}{rgb}{0.0,0.0,0.0}

\definecolor{very-light-gray}{gray}{0.75}
\definecolor{light-gray}{gray}{0.69}
\definecolor{new-blue}{rgb}{0.0,0.0,0.8}
\definecolor{light-blue}{rgb}{0.4,0.45,1}

\usepackage{enumitem}
\setenumerate{label=\upshape(\alph*),itemsep=3pt,topsep=3pt}

\newcommand{\bulletpoint}[1]{\upshape(#1)}

\newcommand{\mycode}{https://doi.org/10.35097/1924}

\definecolor{myBlue3}{RGB}{60,124,155}

\newcommand{\GL}{Ginzburg--Landau\xspace}

\newtheorem{theorem}{Theorem}[section]
\newtheorem{lemma}[theorem]{Lemma}
\newtheorem{corollary}[theorem]{Corollary}
\newtheorem{proposition}[theorem]{Proposition}
\newtheorem{definition}[theorem]{Definition}
\newtheorem{remark}[theorem]{Remark}
\newtheorem{assumption}[theorem]{Assumption}

\newcommand{\Frechet}{Fréchet\xspace}
\newcommand{\wepo}{well-posedness\xspace}

\DeclareMathOperator{\Real}{Re}
\DeclareMathOperator{\Imag}{Im}

\DeclareMathOperator{\vol}{vol}
\newcommand{\ci}{\mathrm{i}}

\newcommand{\kommentare}[1]{}

\newcommand{\dualp}[2]{ \langle #1,#2 \rangle}

\newcommand{\abilmag}[2]{a_A(#1,#2)}

\newcommand{\abilmagstabsym}[2]{\hat{a}_\kappa(#1,#2)}
\newcommand{\ipsymLtwo}[2]{ m ( #1,#2 ) }

\newcommand{\dualHone}{(H^1)'}

\newcommand{\HonekappaSpace}{H^1_\kappa}
\newcommand{\HtwokappaSpace}{H^2_\kappa}
\newcommand{\Honekappa}[1]{\norm{\HonekappaSpace}{#1}}
\newcommand{\Htwokappa}[1]{\norm{\HtwokappaSpace}{#1}}
\newcommand{\Honekappaminus}[1]{\norm{(\HonekappaSpace)'}{#1}}

\newcommand{\LOD}{{\mbox{\rm\tiny LOD}}}

\newcommand{\h}{h}
\newcommand{\testfun}{\varphi}
\newcommand{\testfunTWO}{v}
\newcommand{\testfunTHREE}{w}
\newcommand{\testfunFOUR}{z}
\newcommand{\testfunh}{\varphi_{\h}}

\newcommand{\Ltwoproj}{\pi_{\h}}

\newcommand{\projLtwoisol}{\textup{R}_{\kappa,\h}^{\perp}}
\newcommand{\Ltwotwist}{P_{\ci u}^{\perp}}

\newcommand{\RitzprojLOD}{\textup{R}_{\kappa,\h}^\LOD} 
\newcommand{\projLtwoisolLOD}{\textup{R}_{\kappa,\h}^{\perp,\LOD}}

\newcommand{\errh}{e_{\h}}

\newcommand{\errHmone}{\varepsilon_\h}
\newcommand{\errHmonelin}{\errHmone^{\text{lin}}}
\newcommand{\errHmonenonlin}{\errHmone^{\text{nonlin}}}

\newcommand{\Cbnd}{C_{\text{bnd}}}
\newcommand{\Ccoer}{C_{\text{coe}}}

\newcommand{\Csol}{C_{\textup{sol}}(\sol,\kappa)}
\newcommand{\Csolinv}{C^{-1}_{\textup{sol}}(\sol,\kappa)}
\newcommand{\CsolH}{C_{\textup{sol},\h}(\kappa)}

\newcommand{\Honeperp}{H^1_{\ci \sol}}

\newcommand{\energy}{E}

\renewcommand{\div}{\operatorname{div}}
\newcommand{\curl}{\operatorname{curl}}

\newcommand{\norm}[2]{\lvert| #2 \rvert|_{#1}}

\newcommand{\sol}{u}
\newcommand{\solh}{\sol_{\h}}
\newcommand{\solhn}[1]{\sol_{{\h}_{#1}}}
\newcommand{\solhnTwist}[1]{\widetilde{\sol}_{{\h}_{#1}}}

\newcommand{\solhLOD}{\solh^{\LOD}}

\newcommand{\Nbh}{U} 

\newcommand{\MagF}{A}
\newcommand{\MagFinfty}{A_\infty}
\newcommand{\stabPar}{\beta}

\newcommand{\R}{\mathbb{R}}
\newcommand{\C}{\mathbb{C}}
\newcommand{\VS}{H^1} 
\newcommand{\VSh}{V_{\h}}
\newcommand{\VShLOD}{V_{\h}^{\LOD}}

\newcommand{\VShperp}{V_{\h}^\perp}

\newcommand{\Th}{\mathcal{T}_\h}
\newcommand{\Pone}{\mathcal{P}_1}

\usepackage{geometry}
\begin{document}

\begin{center}
{\LARGE 
Error bounds for discrete minimizers of the Ginzburg--Landau energy
in the high-{\Large$\kappa$} regime\renewcommand{\thefootnote}{\fnsymbol{footnote}}\setcounter{footnote}{0}
 \hspace{-3pt}\footnote{The first author is funded by the Deutsche Forschungsgemeinschaft (DFG, German Research Foundation) -- Project-ID 258734477 -- SFB 1173. The second author also acknowledges the support by the Deutsche Forschungsgemeinschaft through the grant HE 2464/7-1.
Parts of this paper were written while the authors enjoyed the kind hospitality of the Institute for Mathematical Science in Singapore in February 2023 during the program on ``Multiscale Analysis and Methods for Quantum and Kinetic Problems''.}}\\[2em]
\end{center}

\begin{center}
{\large Benjamin D{\"o}rich\footnote[1]{Institute for Applied and Numerical Mathematics, 
	Karlsruhe Institute of Technology, DE-76149 Karlsruhe, Germany.\\ email: \href{mailto:benjamin.doerich@kit.edu}{benjamin.doerich@kit.edu}}
and 
Patrick Henning\footnote[2]{Department of Mathematics, Ruhr-University Bochum, DE-44801 Bochum, Germany.\\ email: \href{mailto:patrick.henning@rub.de}{patrick.henning@rub.de}} }\\[2em]
\end{center}

\begin{center}
{\large{\today}}
\end{center}

\begin{center}
\end{center}

\begin{abstract}
In this work, we study discrete minimizers of the Ginzburg--Landau energy in finite element spaces. Special focus is given to the influence of the Ginzburg--Landau parameter $\kappa$. This parameter is of physical interest as large values can trigger the appearance of vortex lattices. Since the vortices have to be resolved on sufficiently fine computational meshes, it is important to translate the size of $\kappa$ into a mesh resolution condition, which can be done through error estimates that are explicit with respect to $\kappa$ and the spatial mesh width $h$. For that, we first work in an abstract framework for a general class of discrete spaces, where we present convergence results in a problem-adapted $\kappa$-weighted norm. Afterwards we apply our findings to Lagrangian finite elements and a particular generalized finite element construction. In numerical experiments we confirm that our derived $L^2$- and $H^1$-error estimates are indeed optimal in $\kappa$ and $h$.
\end{abstract}

\vspace{12pt}

\vspace{12pt}
	%
	Secondary:
	65N30, 
	35Q56

\section{Introduction}

Superconductors are materials that allow to conduct electricity without any electrical resistance.
Letting 
$\Omega \subset \R^d$,  $d=2,3$, denote a
bounded polyhedral Lipschitz domain occupied by a superconducting material,  
the superconductivity in $\Omega$ can be modeled by a complex-valued wave
function $\sol : \Omega \rightarrow \mathbb{C}$ which is called the order parameter.  
The physical quantity of interest is $|\sol|^2$ which denotes the density of the superconducting electron pairs, where in the appropriate scaling, it holds $0 \leq |\sol|^2 \leq 1$.
This means that the material is not superconducting (in normal state) in $x\in\Omega$ if 
$|\sol(x)|^2  = 0$ 
and behaves like a perfect superconductor if $|\sol(x)|^2  = 1$. In between, different degrees
of superconductivity are possible. Of particular interest are mixed normal-superconducting states where both phases coexist in a lattice of quantized vortices \cite{Abr04}. 

Mathematically, the order parameter can be characterized as a
minimizer of the \GL energy (or Gibbs free energy) given by
\begin{equation} \label{eq:energy_functional_times_kappa2}
	\energy(\sol) = \frac{1}{2} \int_\Omega |\nabla \sol + \ci  \kappa \MagF \sol |^2 + \frac{\kappa^2}{2} \bigl( 1- |\sol|^2 \bigr)^2\,dx,
\end{equation}
where $\MagF\colon \Omega \to \R^d$ is a real-valued magnetic potential and $\kappa$ is the so-called \GL parameter, a material parameter that correlates with the temperature and determines the type of superconductor. By the necessary condition for local extrema, any minimizer $\sol  \in H^1(\Omega)$ must fulfill the condition $E^{\prime}(\sol)=0$, which is known as the \GL equation (GLE) and reads written out (cf. \cite{DuGP92})
\begin{align}
	\label{eq:ginzburg-landau-eqn}
	\Real  \int_\Omega \bigl( \nabla \sol + \ci \kappa \MagF \sol \bigr)  \cdot \bigl( \nabla \testfun + \ci \kappa \MagF \testfun \bigr)^*  
	+
	\kappa^2 \bigl( |\sol|^2 -1 \bigr)  \sol \testfun^* 
	\,dx = 0
	 \,
	  \mbox{ for } \testfun \in H^1(\Omega).
\end{align}
The real-valued magnetic potential $\MagF\colon \Omega \to \R^d$ in the GLE is typically unknown and can be inferred from an external magnetic field $H$ through the condition $H = \curl \MagF$ which is then added as a penalty term to the energy. In this work we consider the simplifying case that $\MagF$ is given, where the focus of our analysis is rather the influence of $\kappa$ on the accuracy of numerical approximations. In fact, the size of the parameter $\kappa$ is crucial for the appearance of vortices \cite{SaS07,SaS12,Ser99,SeS10}. On the one hand, if $\kappa$ is too small, no vortices will appear. On the other hand, the larger the value of $\kappa$, the more vortices appear in the lattice and the more point-like they become \cite{Aftalion99,SaS07}. The so-called high-$\kappa$ regime is hence the physically most interesting regime, but numerically it is also the most challenging one because it requires fine meshes to resolve all lattice structures. This raises an important practical question: how fine do we have to select the mesh size relative to the size of $\kappa$ so that the numerical approximations capture the correct vortex pattern?

Motivated by this question, the main goal of this work is to derive rigorous error bounds for the discrete minimizers with constants that are explicit and optimal in the spatial parameter $\h$ and the \GL parameter $\kappa$.

The first work where the approximation properties of discrete solutions to the stationary GLE were analyzed is the seminal SIAM Review article by Du, Gunzburger and Peterson \cite{DuGP92} (see also \cite{DuGP93} for periodic boundary conditions). The paper considers $H^1$-error estimates in finite element (FE) spaces for both the order parameter $u$ and the magnetic potential $\MagF$. The proof technique considers fixed (compact) intervals of $\kappa$-values and does not trace all $\kappa$-dependencies that enter through the size of these intervals and through uniform bounds for certain operator norms (that are linked to the chosen interval). The proof also considers a modified setting where $\energy^{\prime\prime}(u)$ is assumed to have a trivial kernel. However, the solutions to the GLE \eqref{eq:ginzburg-landau-eqn} can be only locally unique up to gauge transformations \cite{DuGP92}. In our case, these transformations are of the form $\sol \mapsto \exp( - \ci \theta) \sol$ for any $\theta\in\mathbb{R}$. In fact, it is easily seen that $\energy(\sol)=\energy( \exp( - \ci \theta) \sol)$ for all such $\theta$, which hence leads to a cluster of (qualitatively equivalent) solutions $\exp( - \ci \theta) \sol$. In turn, we have $\dualp{\energy^{\prime\prime}(\sol) \,  \exp( - \ci \tfrac{\pi}{2} ) \sol}{ \cdot } = 0$ which shows that $\energy^{\prime\prime}(\sol)$ can become singular. Hence, it makes sense to revisit the results \cite{DuGP92} with new proof techniques that allow us to follow all $\kappa$-dependencies and which allow us to avoid an assumption of local uniqueness. To the best of our knowledge, there are only two other works that address the convergence of discrete solutions to the stationary GLE: In \cite{DuNicolaidesWu98} a spatial discretization based on a covolume method is suggested, and in \cite{QuJu05} a finite volume discretization is used to solve the GLE on the sphere. 
In both papers the convergence of a subsequence of discrete solutions to a continuous minimizer is established, however without rates in $\h$ and $\kappa$.

With the goal to close this gap in the literature for finite element discretizations, our error analysis is performed in a general framework of FE methods, where we state our results under natural assumptions on the discrete spaces.
We first establish bounds on the discrete minimizers which are explicit in $\kappa$ and independent of $\h$.
This enables us to provide an abstract convergence result which identifies a suitable, continuous minimizer
of \eqref{eq:energy_functional_times_kappa2}. This a priori information is crucial in the derivation of the error bounds.
In order to exploit the structure of the problem, we have to study the properties of the second \Frechet derivative
of the energy $\energy$.
In particular, we carry over the inf-sup stability to our discrete setting under a smallness condition related to the product $\kappa \h$.
Let us emphasize that this is not a technical issue, but is indeed observed in our numerical experiments.
We employ a problem adapted
scalar product and its Ritz projection, which captures the
one-dimensional kernel of $\energy''$, to extract optimal error bounds not only for the $H^1$-norm, but also
new error bounds
for the $L^2$-norm and the energy. Our numerical experiments confirm that the predicted scaling of the
error in $\kappa$ and $\h$ is asymptotically sharp.

It is worth to mention that, aside from stationary Ginzburg--Lindau equations, there has been a lot of work on the numerical analysis of the time-dependent problem that describes the dynamics of superconductors, where we exemplarily refer to \cite{Chen97,CheD01,Du94b,Du94,Du97,DuGray96,DuanZhang22,GJX19,GaoS18,Li17,LiZ15,LiZ17} and the references therein. For works with a particular emphasis on tracing the influence of $\kappa$ in the estimates, we refer to \cite{Bartels2005,Bar05,BarMO11} for the case of vanishing vector potentials $\MagF$. Due to the different nature of the time-dependent problem, we will not discuss the equation any further here.

\smallskip

The rest of the paper is organized as follows:
In Section~\ref{sec:framework}, we introduce the analytical framework and present some results on continuous minimizers of \eqref{eq:energy_functional_times_kappa2}. In particular, we discuss the assumptions concerning uniqueness of minimizers.
For an abstract finite element space discretization, we present in Section~\ref{sec:main} our main results on the existence, boundedness, and approximation of discrete minimizers. An application to linear Lagrange finite elements is also given.
Numerical experiments which illustrate our theoretical findings
and confirm the convergence rates as well as the $\kappa$-dependency of our bounds are shown in Section~\ref{sec:num_exp}. 
The proofs of our main results are given in Section~\ref{sec:proofs}. Finally, in Section~\ref{sec:lod} we present a nonstandard application of the abstract result to spaces based on the Localized Orthogonal Decomposition.

\subsection*{Notation} For a complex number $z \in \mathbb{C}$, we use $z^*$ for the complex conjugate of $z$. In the whole paper we further denote by $L^2:=L^2(\Omega,\mathbb{C})$ the Hilbert space of $L^2$-integratable complex functions, but equipped with the {\it real} scalar product $\ipsymLtwo{\sol}{v} :=\Real \int_{\Omega} v \, w^* \,dx$ for $v,w\in L^2$. Hence, we interpret the space as a {\it real} Hilbert space. Analogously, we equip the space $H^1:=H^1(\Omega,\mathbb{C})$ with the scalar product $\ipsymLtwo{v}{w}+\ipsymLtwo{\nabla v}{\nabla w}$. This interpretation is crucial so that the Fr\'echet derivatives of $E$ are meaningful and exist on $H^1$. For any space $X$, we denote its dual space by $X'$. Note that this implies, that the elements of the dual space of $H^1$ consist of real-linear functionals, which are not necessarily complex-linear. For example, if $F(v):=m(f,v)$ for some $f \in L^2$, then it holds $F(\alpha \, v)=\alpha\, F(v)$ if $\alpha \in \mathbb{R}$, but in general {\it not} if $\alpha \in \mathbb{C}$.

In the following $C$ will denote a generic constant which is independent of $\kappa$ and the spatial mesh parameter $h$, but might depend on numerical constants as well as $\Omega$ and $\MagF$. 
In particular, we will write $\alpha \lesssim \beta$ if there is a constant $C$ independent of $\kappa$ and $h$ such that $\alpha \leq C \,\beta $.

\section{Analytical framework}
\label{sec:framework}

In this section, we present several results concerning the continuous minimizers of \eqref{eq:energy_functional_times_kappa2}.

From now on, we assume that the magnetic potential $\MagF$ satisfies 
\begin{equation} \label{eq:ass_\MagF_for_H2}
	\MagF \in L^\infty(\Omega,\mathbb{R}^d),
	\qquad
	\div \MagF = 0 \text{ in } \Omega,
	\qquad 
	\MagF \cdot \nu =0  \text{ on } \partial \Omega.
\end{equation}
The above assumption can be rigorously justified on convex and on smooth domains $\Omega$, cf. \cite{DuGP92}. However, we also note that the $L^\infty$-regularity of $A$ might not be available on general complex geometries with re-entrant corners.

Further, we introduce the dual pairing $\dualp{\sol }{ \testfun } \coloneqq \dualp{\sol }{ \testfun }_{\dualHone, H^1}$, and the bilinear forms
given by
\begin{align} 
	\ipsymLtwo{\sol}{\testfun} \coloneqq \Real \int_{\Omega} \sol \testfun^* \,dx ,
	\quad
	\abilmag{\sol}{\testfun} \coloneqq  \Real \int_\Omega \bigl( \nabla \sol + \ci \kappa \MagF \sol \bigr) \cdot  \bigl( \nabla \testfun + \ci \kappa \MagF \testfun \bigr)^*  \,dx ,
	\label{eq:def_forms_abil}
\end{align}
as well as the norm $\norm{H^1}{\testfun}^2 \coloneqq \norm{L^2}{\nabla \testfun }^2 + \norm{L^2}{ \testfun }^2$, the scaled norms 
\begin{equation} \label{eq:def_norms}
	\Honekappa{\testfun}^2 \coloneqq \norm{L^2}{\nabla \testfun }^2 + \kappa^2 \norm{L^2}{ \testfun }^2, 
	\qquad
	\Htwokappa{\testfun}^2 \coloneqq \norm{H^2}{\testfun}^2 + \kappa^2 \Honekappa{\testfun}^2,
\end{equation}
and the induced norm $\Honekappaminus{f} = \sup_{\testfun \in H^1} \frac{f(\testfun)}{\Honekappa{\testfun}}$.
We abbreviate $\MagFinfty = \norm{L^\infty}{\MagF}$,
and define the stabilized inner product on $\VS = H^1(\Omega)$ for $\sol,\testfun\in \VS$ by
\begin{equation} \label{eq:def_abilmagstabsym}
	\abilmagstabsym{\sol}{\testfun}  \coloneqq \abilmag{\sol}{\testfun} + \stabPar^2 \ipsymLtwo{\sol}{\testfun}_{L^2} ,
	\quad \text{with } \stabPar^2 = \kappa^2 (\MagFinfty^2 +1).
\end{equation}
We call it stabilized since this modification enables us to show boundedness and coercicvity of $\abilmagstabsym{\cdot}{\cdot}$ with respect to the
$\HonekappaSpace$-norm defined in \eqref{eq:def_norms}.

\begin{lemma}
	\label{lem:prop_bil}
	There are $\kappa$-independent constants $\Cbnd,\Ccoer>0$ such that
	for all $\testfunTWO,\testfun\in \VS$
	\begin{align}
		\abilmagstabsym{\testfunTWO}{\testfun} &\leq \Cbnd \, \Honekappa{\testfunTWO} \Honekappa{\testfun},
		\qquad
		\text{and}
		\qquad
		\abilmagstabsym{\testfun}{\testfun} \geq  \Ccoer \, \Honekappa{\testfun}^2.
	\end{align}
\end{lemma}

\begin{proof}
	The boundedness is a straightforward application of the Cauchy-Schwarz inequality. For the coercivity, we note that by Young's inequality it holds
	\begin{equation}
		|\nabla \testfun + \ci \kappa \MagF \testfun|^2 
		\geq
		|\nabla \testfun|^2 - 2   |\nabla \testfun| |\kappa \MagF \testfun| + |\kappa \MagF \testfun|^2
		\geq 
		\frac12 |\nabla \testfun|^2 - \kappa^2 \MagFinfty^2 |\testfun|^2.
	\end{equation}
	By the choice of $\beta$, we conclude the lower bound.
\end{proof}

A straightforward calculation shows that the
energy is  (real-)\Frechet differentiable and satisfies for all $\testfun\in H^1$
\begin{align} \label{eq:first_Frechet}
	\dualp{\energy'(\sol)}{\testfun}
	&= 
	\Real  \int_\Omega \bigl( \nabla \sol + \ci \kappa \MagF \sol \bigr)  \cdot \bigl( \nabla \testfun + \ci \kappa \MagF \testfun \bigr)^*  
	+
	\kappa^2 \bigl( |\sol|^2 -1 \bigr)  \sol \testfun^* 
	\,dx .
\end{align}
In particular any minimizer $\sol \in H^1$ satisfies $\energy'(\sol) = 0$.
The natural boundary condition is given by $\bigl( \nabla \sol + \ci \kappa \MagF \sol \bigr)  \cdot \nu =0$ on $\partial \Omega$. Since $\MagF$ has a vanishing trace by assumption, the natural boundary condition collapses to the standard homogenous Neumann condition  $\nabla \sol \cdot \nu =0$ on $\partial \Omega$.
Our first result collects the existence of a minimizer $\sol$ and its properties.

\begin{theorem} \label{thm:cont_minimizer}
	For every $\kappa \geq 0$ there exists a minimizer $\sol \in H^1$ of  \eqref{eq:energy_functional_times_kappa2}.
	Further, any minimizer fulfills
	\begin{align}
		|\sol(x)| &\leq 1 \,\mbox{ for all } x\in \Omega ,
		\qquad 
		\Honekappa{\sol}\lesssim \kappa , 
		\intertext{and if $\Omega$ is convex then $\sol \in H^2$ and}
		\norm{L^4}{\nabla \sol} &\lesssim \kappa,
		\qquad
		\Htwokappa{\sol} \lesssim \kappa^2,
	\end{align}
	where the hidden constants in the above estimates are independent of $\kappa$ and $\sol$.
\end{theorem}

\begin{proof}
	First note that the energy $\energy$ is continuous in $H^1(\Omega)$, and further weakly lower semi-continuous, see e.g., \cite[Thm.~1.6]{Struwe08}. In addition, a simple calculation shows
	\begin{equation}
		\energy(\sol) = \frac{1}{2}  \abilmagstabsym{\sol}{\sol} 
		+
		\frac{\kappa^2}{4}  \int_\Omega
		\bigl( 1 + \frac{\stabPar^2}{\kappa^2} - |\sol|^2 \bigr)^2  
		+  
		1
		-
		\bigl( 1 + \frac{\stabPar^2}{\kappa^2} \bigr)^2 \,dx,
	\end{equation}
	and hence $\energy(\sol) \to \infty$ as $\Honekappa{\sol} \to \infty$.
	The standard arguments then imply the existence of a minimizer, see e.g., \cite[Thm.~1.2]{Struwe08}.
	For the pointwise bound, we refer to \cite[Prop.~3.11]{DuGP92}, which implies a bound in $L^2$ independent of $\kappa$.  We further have 
	\begin{equation}
		\norm{L^2}{\nabla \sol} \leq \norm{L^2}{\nabla \sol + \ci \kappa \MagF \sol} + \kappa \MagFinfty \norm{L^2}{\sol} \lesssim \energy(0)^{1/2} + \kappa  \lesssim \kappa .
	\end{equation}
	Since $\energy'(\sol) = 0  $, we rearrange to
	\begin{align}
		\abilmag{\sol}{\testfun} =
		- \kappa^2 \Real \int_\Omega \bigl( |\sol|^2 -1 \bigr)  \sol \testfun^* 
		\,dx
		= \ipsymLtwo{ f }{\testfun} 
	\end{align}
	with $\norm{L^2}{f} \lesssim  \kappa^2 $, and obtain with \eqref{eq:ass_\MagF_for_H2}
	\begin{equation} 
		\Real \int_{\Omega}
		\nabla \sol \cdot \nabla \testfun^* 
		\,dx
		+ \Real \int_{\Omega}
		\sol \, \testfun^* 
		\,dx 
		= \ipsymLtwo{ f }{\testfun} 
		-
		\Real \int_{\Omega}
		\bigl(
		- 2 \ci \kappa \MagF \cdot \nabla \sol 
		+ \kappa^2 |\MagF|^2 \sol + \sol \bigr) \testfun^*
		\,dx .
	\end{equation}
	If $\Omega$ is convex, standard elliptic regularity theory for the Poisson problem with homogenous Neumann boundary condition (cf. \cite[Theorem 3.2.1.3]{Grisvard}) gives us $\sol \in H^2$ with 
	\begin{equation}
		\norm{H^2}{\sol} \lesssim  \norm{L^2}{f} + (1+\kappa^2) \norm{L^2}{\sol} + \kappa \norm{L^2}{\nabla \sol} 
		\lesssim \kappa^2,
	\end{equation}
	where we used the $L^2$- and $H^1$-bounds for $\sol$ in the last step. 
	Finally, we have
	\begin{equation}
		\norm{L^4}{\nabla \sol}^4
		=
		\int_\Omega \nabla \sol \cdot \nabla \sol^* |\nabla \sol|^2 \,dx 
		= 
		- \int_\Omega \sol \div \bigl(\nabla \sol^* |\nabla \sol|^2  \bigr) \,dx 
		\lesssim 
		\norm{L^\infty}{\sol}
		\norm{H^2}{\sol}
		\norm{L^4}{\nabla \sol}^2,
	\end{equation}
	which yields the last estimate.
\end{proof}

Since $u$ is a global minimizer of the energy $\energy$, it must not only hold $\langle E^{\prime}(u) , \testfun\rangle =0$ but also $\langle E^{\prime\prime}(u) \testfun , \testfun \rangle \ge 0$ for all $\testfun\in H^1$. Later we will make use of these conditions. For that we require a corresponding representation of the second \Frechet derivative of $\energy$. This and its properties are summarized in the following lemma.

\begin{lemma} \label{lem:Frechet_functional}

\begin{itemize}
	\item[{\normalfont (a)}] The energy is twice (real-)\Frechet differentiable and satisfies for $\testfun,\testfunTWO \in H^1$
	\begin{align}
		\dualp{\energy''(\sol) \testfunTWO}{\testfun}	 
		=  
		\Real  \int_\Omega 
		&\bigl( \nabla \testfunTWO + \ci \kappa \MagF \testfunTWO \bigr)  \cdot \bigl( \nabla \testfun + \ci \kappa \MagF \testfun \bigr)^*   
		\\
		&+
		\kappa^2 
		\bigl(  ( |\sol|^2 -1  ) \testfunTWO \testfun^*  + \sol^2 \testfunTWO^* \testfun^* + |\sol|^2 \testfunTWO \testfun^* \bigr)
		\,dx .
	\end{align}
	\item[{\normalfont (b)}] 
	For  $\testfun,\testfunTWO\in H^1$ it holds
	\begin{equation} \label{eq:E_prime_prime_sym}
		\dualp{\energy''(\sol) \testfunTWO}{\testfun}	 = \dualp{\energy''(\sol) \testfun}{\testfunTWO}	
		\quad
		\text{and}
		\quad
		|  \dualp{\energy''(\sol)  \testfunTWO}{\testfun}  | 
		\lesssim  \Honekappa{\testfunTWO} \Honekappa{\testfun} .
	\end{equation}
\end{itemize}
\end{lemma} 

\begin{proof}
	The \Frechet derivative is computed in a straightforward manner, and the symmetry follows from the representation by noting the real part in front of the integral. For the bound, we employ Lemma~\ref{lem:prop_bil} as well as $|\sol| \leq 1$.
\end{proof}

As explained in the introduction, a minimizer of \eqref{eq:energy_functional_times_kappa2} cannot be unique due to the invariance of the energy under complex rotations, i.e. if $\sol$ is a minimizer, then $ \exp( -\ci \phi) \sol$ is also a minimizer for any $\phi \in \R$. This property is known as gauge invariance and the mapping $\sol \mapsto \exp( -\ci \phi) \sol$ is called a gauge transformation. On polygonal domains, minimizers are believed to be locally unique, that means, that in a sufficiently small environment of $\sol$, the only other minimizers are exactly the ones obtained by gauge transformations. However, a general proof for this local uniqueness property is one of the great challenges of the field and has not yet been established. For partial results in various important settings, we refer to \cite{PacRiv00,Riv2002,SaS07,WeiWu2020} and the references therein.

In the following, we hence make the local uniqueness an (reasonable) assumption for our analysis. Furthermore, we later describe how to check the validity of the assumption numerically for any given setting so that we explicitly know if it is fulfilled. 

In the definition below we express the local uniqueness by curves passing through a minimizer. To be precise, we look at the energy level $\ell: t \mapsto E(\, \gamma(t) \, )$ for a smooth curve $\gamma : \mathbb{R} \rightarrow H^1$ with $\gamma(0)=\sol$. If $\sol$ is a minimizer of $E$, then $\ell(t)$ has a local minimum at $t=0$ (i.e. $\ell^{\prime}(0)=0$ and $\ell^{\prime\prime}(0)\ge0$). Furthermore, $\sol$ is locally unique up to gauge, if $\ell^{\prime\prime}(0)>0$ for all directions $\gamma^{\prime}(0)$ that are not parallel to $\ci \sol$. Note that the direction $\ci \sol$ is the tangent in $\sol$ on the circle line $t \mapsto \exp( -\ci t) \sol$, i.e., the direction in which the value of the energy does not change, cf. Figure \ref{illustration-tangent}. Since the energy is constant on the circle line, we naturally have $\ell^{\prime\prime}(0)=0$ in this direction. The following definition formalizes this type of local uniqueness. 

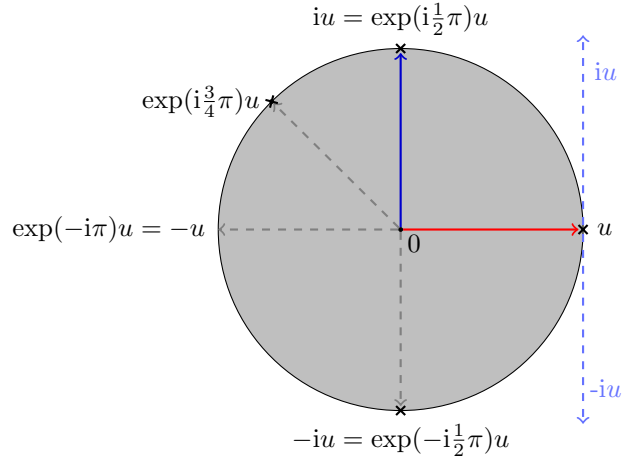
\begin{figure}
	\centering
	\begin{tikzpicture}  %
		[scale=0.6, point/.style = {draw, circle, fill=black, inner sep=0.5pt}]
		\filldraw[fill=light-gray,color=light-gray]  (0,0) node[]{} -- (4,0) node[]{} -- (0,4) node[]{}  -- cycle;
		\filldraw[fill=very-light-gray] (4,0) arc[start angle=0, end angle=360, radius=4]; 
		\node (C) at (0,0) [point]{}; 
		\node (iU) at (0,4) [cross,label=above:{$\ci \sol = \exp( \ci \tfrac{1}{2}\pi)\sol$}]{};
		\node (minusiU) at (0,-4) [cross,label=below:{$-\ci \sol = \exp(- \ci \tfrac{1}{2}\pi)\sol$}]{};
		\node (U) at (4,0) [cross,label=right:$\sol$]{};
		\node (E1) at (4, 4.5) []{}; 
		\node (E2) at (4, -4.5) []{};  
		\node (Utheta) at (-2.828427124748,2.828427124748) [point,label=left:{{$\exp( \ci \tfrac{3}{4}\pi)\sol$}}]{}; 
		\draw[->,thick, color=red] (C) -- (U);
		\draw[->,thick,color=new-blue] (C) -- (iU);
		\draw[->,thick,dashed,color=gray] (C) -- (minusiU);
		\draw[->,thick,dashed,color=gray] (C) -- (-4,0);
		\draw[->,thick, dashed,color=light-blue] (U) -- (E1);
		\draw[->,thick, dashed,color=light-blue] (U) -- (E2);
		\draw[->,thick, dashed,color=gray] (C) -- (Utheta);
		\node at (Utheta) [cross,rotate=30]{};
		\node at (4.5, 3.5) {{\color{light-blue}$\ci \sol$}};
		\node at (4.5, -3.5) {{\color{light-blue}-$\ci \sol$}};
		\node at (0.3, -0.3) {$0$};
		\node at (-6.4,0) {{$\exp(- \ci \pi)\sol=-\sol$}};
	\end{tikzpicture}
	\caption{For a given minimizer $\sol$, the figure illustrates the circle line parametrized by the $2\pi$-periodic curve $\gamma : t\mapsto \exp(- \ci t) \sol$ for $t \in [-\pi ,\pi)$. The tangent direction in $u$ is given by $\gamma^{\prime}(0) = \ci \sol$ and the energy $E$ is constant and minimal on the whole circle line, i.e. $\frac{\mbox{\rm d}^2}{\mbox{\rm d}t^2} E(\hspace{1pt}\gamma(t)\hspace{1pt})= 0$.}
	\label{illustration-tangent}
\end{figure}

\begin{definition}[Local uniqueness up to gauge transformation]
	\label{def-local-uniqueness}
	Let $\sol$ be a minimizer of \eqref{eq:energy_functional_times_kappa2}. 
	We call $u$ a {\rm locally unique minimizer up to gauge transformation} if, for all 
	smooth curves $\gamma : [-\pi,\pi) \rightarrow H^1$ with $\gamma(0)=\sol$, it holds 
	\begin{eqnarray*}
		\frac{\mbox{\rm d}^2}{\mbox{\rm d}t^2} E(\, \gamma(t) \, )\vert_{t=0} = 0
		\quad \Longleftrightarrow \quad 
		\gamma^{\prime}(0) \in \mbox{\rm span} \{ \ci \sol \} := \{ \alpha \, \ci \sol \, | \, \alpha \in \mathbb{R} \}.
	\end{eqnarray*}
	Note that $u$ being a minimizer always implies that
	$\frac{\mbox{\rm\tiny d}}{\mbox{\tiny\rm d}t} E(\, \gamma(t) \, )\vert_{t=0} = 0$
as well as
$\frac{\mbox{\tiny\rm d}^2}{\mbox{\tiny\rm d}t^2} E(\, \gamma(t) \, )\vert_{t=0} \ge 0 $, independent if it is locally unique or not.
\end{definition}
From now on, we assume local uniqueness for our error analysis. %
\begin{assumption} \label{ass:cinfsup}
	The minimizers $\sol$ of the Ginzburg-Landau energy \eqref{eq:energy_functional_times_kappa2} are locally unique up to gauge transformation in the sense of Definition~$\ref{def-local-uniqueness}$. 
\end{assumption}
For any minimizer $\sol$, the above assumption implies inf-sup stability of $E^{\prime\prime}(\sol)$ on the $\ipsymLtwo{\cdot}{\cdot}$-orthogonal complement of $\mbox{\rm span} \{ \ci \sol \}$ in $H^1$. To formalize and prove this result, we define the corresponding space by
\begin{equation}
	\Honeperp \coloneqq H^1 \cap ( \ci \sol)^\perp \coloneqq \{ \testfun \in H^1 \, | \, m(\ci \sol , \testfun ) = 0 \}.
\end{equation}
Note that $\Honeperp$ is a closed subspace of $H^1$ and that our error analysis will be naturally restricted to it. The claimed inf-sup stability is specified in the following proposition.
\begin{proposition}
	\label{proposition-inf-sup-stability}
	Let $\sol$ be a minimizer of \eqref{eq:energy_functional_times_kappa2} and let Assumption \ref{ass:cinfsup} be fulfilled, i.e., $\sol$ is locally unique up to gauge transformation. Then, there is a constant $\Csol \gtrsim 1$ such that
	\begin{equation} \label{eq:inf_sup_condition_cont}
		\Csolinv  \leq \inf_{\testfunTWO \in \Honeperp} \sup_{\testfun\in \Honeperp} \frac{\dualp{\energy''(\sol)  \testfunTWO}{\testfun}}{\Honekappa{\testfunTWO} \Honekappa{\testfun}  } .
	\end{equation}
	Furthermore, $\energy''(\sol) $ is singular in the direction $\ci \sol$, i.e.,
	\begin{eqnarray*}
		\dualp{\energy''(\sol) \, \ci \sol}{\testfun} = 0 \qquad \mbox{for all } \testfun \in H^1.
	\end{eqnarray*}
\end{proposition}
Since Assumption \ref{ass:cinfsup} implies a positive spectrum of $E''(\sol)$ on $\Honeperp$, the inf-sup stability follows from the G{\aa}rding inequality. For completeness we elaborate the short proof below, also with the purpose to emphasize how to check the local uniqueness numerically.
\begin{proof}
	Let $v \in \Honeperp \setminus \{ 0\}$ be arbitrary and $\gamma : [-\pi,\pi)\rightarrow H^1$ a smooth curve with $\gamma(0)=u$ and $\gamma^{\prime}(0)=v$. Then Assumption \ref{ass:cinfsup} implies
	\begin{equation}
		\label{positivityEprimeprime}
		0 
		<
		 \frac{\mbox{\rm d}^2}{\mbox{\rm d}t^2} E(\, \gamma(t) \, ) \vert_{t=0} 
		= 
		\langle E^{\prime}( \gamma(0) ) ,  \gamma^{\prime\prime}(0)  
		\rangle
		+
		\langle E^{\prime\prime}( \gamma(0) ) \, \gamma^{\prime}(0), \gamma^{\prime}(0) \rangle
		= 
		\langle E^{\prime\prime}(\sol ) v, v \rangle.
	\end{equation}
	Hence, the eigenvalue problem seeking $(w_{\ell}, \lambda_{\ell}) \in H^1 \times \mathbb{R}$ with
	\begin{eqnarray}
		\label{eigenvalues-secvarE}
		\langle E^{\prime\prime}(u) w_{\ell} , v \rangle = \lambda_{\ell} \, m( w_{\ell} , v ) \quad \mbox{for all } v \in H^1
	\end{eqnarray}
	has a non-negative spectrum. A direct calculation shows that the smallest eigenvalue is $\lambda_1= 0$ with eigenfunction $w_1 = \ci \sol$. Furthermore, since $\langle E^{\prime\prime}(\sol ) \cdot ,\cdot \rangle$ is a symmetric bilinear form on $H^1$, the second smallest eigenvalue $\lambda_2$ is positive due to the Courant--Fischer theorem which yields
	$$
	\lambda_2 = \underset{v \not=0}{\inf_{v\in \Honeperp}}  \frac{\langle E^{\prime\prime}(u) v , v \rangle}{m(v,v)} \overset{\eqref{positivityEprimeprime}}{>} 0.
	$$
	Observe that the strict positivity of the second eigenvalue of $E^{\prime\prime}(u)$ is in fact equivalent to the local uniqueness of $\sol$ in the sense of Definition \ref{def-local-uniqueness}. Hence, by computing $\lambda_2$ we can practically verify if a computed minimizer $u$ is locally unique.
	
	Thanks to $\lambda_2>0$, we obtain that $E^{\prime\prime}(u)$ is injective on the $m(\cdot,\cdot)$-orthogonal complement of the first eigenspace, i.e., on $\Honeperp$. Furthermore, it is easily seen that the following G{\aa}rding inequality holds for all $v\in H^1$ (and in particular $v\in \Honeperp$):
	\begin{eqnarray}
		\label{gardinginequality}
		\langle E^{\prime\prime}(u) v , v \rangle \ge \frac{1}{2} \| v \|^2_{H^1} - (1+\kappa^2) (1 + A_{\infty}^2 ) \| v \|^2_{L^2}.
	\end{eqnarray}
	It is well known (cf. \cite{GaticaHsiao94,Yosida1980}) that the injectivity on $\Honeperp$ and the G{\aa}rding inequality (in combination with the Lax-Milgram theorem) imply the Fredholm alternative for $E^{\prime\prime}(u) : \Honeperp \rightarrow (\Honeperp)'$, i.e., $E^{\prime\prime}(u)$ has a bounded inverse on the orthogonal complement of $\ci \sol$. With this insight, let $v \in \Honeperp$ be arbitrary and let $z \in \Honeperp$ be the unique solution to 
	\begin{eqnarray*}
		\langle E^{\prime\prime}(u) z , w \rangle = \mu \, m(v, w)\qquad \mbox{for all } w\in \Honeperp
	\end{eqnarray*}
	for some constant $\mu>(1+\kappa^2) (1 + A_{\infty}^2 )$. Since  $E^{\prime\prime}(u)$ has a bounded inverse it also holds
	\begin{eqnarray}
		\label{bounded-inverse-estimate-z}
		\| z \|_{H^1} \le C \, \mu \, \| v \|_{H^1} 
	\end{eqnarray}
	for some constant $C=C(u,\kappa)$. In conclusion we obtain
	\begin{eqnarray*}
		\langle E^{\prime\prime}(u) (v+z) , v \rangle
		= \langle E^{\prime\prime}(u) v, v \rangle + \mu \, m(v, v) 
		\overset{\eqref{gardinginequality}}{\ge} \frac{1}{2} \| v\|_{H^1}^2 
		\overset{\eqref{bounded-inverse-estimate-z}}{\ge} C \, \| v\|_{H^1} \| v+z \|_{H^1}
	\end{eqnarray*}
	with $C=C(u,\kappa,\mu)>0$. Hence, 
	\begin{eqnarray*}
		\sup_{\phi \in \Honeperp}  \frac{\langle E^{\prime\prime}(u) \phi , v \rangle }{\| \phi \|_{H^1} \| v \|_{H^1}}
		\ge \frac{\langle E^{\prime\prime}(u) (v+z) , v \rangle}{\| v+z\|_{H^1} \| v \|_{H^1}}
		\ge C.
	\end{eqnarray*} 
	Since $v \in \Honeperp$ was arbitrary, the desired inf-sup stability follows from the symmetry of $\langle E^{\prime\prime}(u) \, \cdot , \cdot\rangle$ and the equivalence of the $H^1$-norm and the $H^1_{\kappa}$-norm (up to $\kappa$-constants).
\end{proof}

\begin{remark}[Verifying the local uniqueness numerically]
	\label{rem:loc_uniq}
	The proof of Proposition \ref{proposition-inf-sup-stability} shows that Assumption \ref{ass:cinfsup} is fulfilled for a minimizer $u$ if and only if the second smallest eigenvalue, $\lambda_2$, of 
	\begin{eqnarray*}
		\langle E^{\prime\prime}(u) w_{\ell} , v \rangle = \lambda_{\ell} \, m( w_{\ell} , v ) \quad \mbox{for all } v \in H^1
	\end{eqnarray*}
	is positive. Given a computed discrete minimizer $u_h$ that approximates $u$, we can check this condition numerically. Since $\| u_h - u \|_{H^1} \le \varepsilon(h) \rightarrow 0$ independent of local uniqueness (cf. Proposition \ref{prop:abstract_convergence} below), the second smallest eigenvalue of $E^{\prime\prime}(u_h)$ converges to the second smallest eigenvalue of $E^{\prime\prime}(u)$. Hence, if $\lambda_2$ is clearly bounded away from zero, $u$ must be locally unique in the sense of Definition \ref{def-local-uniqueness}. The practical assembly of the bilinear form $\langle E^{\prime\prime}(u_h) \, \cdot ,\cdot \rangle$ is discussed in Appendix~\ref{appendix}. In our numerical experiments we present the corresponding values for $\lambda_2$ and we observed that the local uniqueness up to gauge transformations was always fulfilled.
\end{remark}

In the next step, we will derive stability and regularity estimates for solutions to variational problems on $\Honeperp$. The variational problems will later play a crucial role in our error analysis and involve the stabilized bilinear form $\abilmagstabsym{\cdot}{\cdot}$, as well as the inf-sup stable bilinear form $\langle E^{\prime}(u) \, \cdot , \cdot \rangle$.

\begin{lemma} \label{lem:wepo_abilmagstab}
	For any $f \in L^2(\Omega)$, there is $\testfunFOUR\in \Honeperp \subset H^1(\Omega)$ such that
	\begin{equation}
		\abilmagstabsym{\testfunFOUR}{\testfun} =  \ipsymLtwo{f}{\testfun}, \quad  \text{ for all } \testfun \in \Honeperp ,
	\end{equation}
	and there hold the bounds
	\begin{equation}
		\Honekappa{\testfunFOUR} 
		\lesssim \Honekappaminus{f}
		\lesssim \frac{1}{\kappa} \norm{L^2}{f}
		\, \, 
		\mbox{and, if $\Omega$ is convex, then $\testfunFOUR \in H^2$ and} 
		\
		\Htwokappa{\testfunFOUR} \lesssim  \norm{L^2}{f},
	\end{equation}
	where the (hidden) constants in the bounds are independent of $\kappa$. 
\end{lemma}

\begin{proof}
	Since $\abilmagstabsym{\cdot}{\cdot}$ is still coercive on $\Honeperp$, we immediately obtain the unique solution, and also the bounds in $\HonekappaSpace$. 
	Furthermore, we have for any $f \in L^2$ that
	\begin{equation} \label{eq:relation_Honeminus_L2}
		\Honekappaminus{f} 
		=
		\sup_{  \Honekappa{\testfun} = 1} \ipsymLtwo{f}{\testfun}
		\leq
		\sup_{  \Honekappa{\testfun} = 1}  \frac{1}{\kappa} \norm{L^2}{f} \kappa \norm{L^2}{\testfun}
		\leq  \frac{1}{\kappa} \norm{L^2}{f} ,
	\end{equation}
	which yields the second inequality. For the bound in the $\HtwokappaSpace$-norm for convex domains, let $\testfun \in H^1$ and decompose as $\testfun = \widehat{\testfun} + \alpha  (\ci \sol) $
	with $\widehat{\testfun} \in \Honeperp$ and
	$\alpha = \ipsymLtwo{\testfun}{\ci \sol} \norm{L^2}{\sol}^{-2} $. Then,
	\begin{align}
		\abilmagstabsym{\testfunFOUR}{\testfun} &= 	\abilmagstabsym{\testfunFOUR}{\widehat{\testfun} } +  \alpha \, \abilmagstabsym{\testfunFOUR}{ \ci \sol }
		=  \ipsymLtwo{f}{\widehat{\testfun} }  +  \alpha \, \abilmagstabsym{\testfunFOUR}{ \ci \sol }
		\\
		&=  \ipsymLtwo{f}{\testfun}  -
		\alpha \, \ipsymLtwo{f}{\ci \sol }+  \alpha \, \abilmag{\testfunFOUR}{ \ci \sol },
	\end{align}
	where we used \eqref{eq:def_abilmagstabsym} in the last step. We first note
	\begin{equation}
		| \ipsymLtwo{f}{\testfun}  -
		\alpha \, \ipsymLtwo{f}{\ci \sol } | \leq 2 \norm{L^2}{f} \norm{L^2}{\testfun} \,,
	\end{equation}
	and then employ $\energy'(\ci \sol) = 0$ to obtain
	\begin{align}
		| \abilmag{\testfunFOUR}{ \ci \sol } | &=
		| \dualp{\energy'(\ci \sol)}{\testfunFOUR} -
		\kappa^2 	\Real  \int_\Omega \bigl( |\sol|^2 -1 \bigr) \ci \sol \testfunFOUR^* 
		\,dx |
		\leq \kappa^2 \norm{L^2}{\sol} \norm{L^2}{\testfunFOUR} \lesssim \norm{L^2}{f},
	\end{align}
	where we exploited $ \kappa^2 \norm{L^2}{\testfunFOUR} \lesssim \norm{L^2}{f}$ in the last line. Altogether we have shown that there exists some $f_\testfunFOUR\in L^2$ such that it holds for all $\testfun \in H^1$
	\begin{align} \label{eq:var_form_extended_test_fun}
		\abilmagstabsym{\testfunFOUR}{\testfun} &=  \ipsymLtwo{f_\testfunFOUR}{\testfun}  ,\quad \norm{L^2}{f_\testfunFOUR} \lesssim \norm{L^2}{f}.
	\end{align}
	We conclude as in Theorem~\ref{thm:cont_minimizer}: We write
	\begin{equation} \label{eq:expansion_ahat}
		\abilmagstabsym{\testfunFOUR}{\testfun} 
		=  	  
		\Real \int_{\Omega}
		\nabla \testfunFOUR\cdot \nabla \testfun^*  +  \stabPar^2 \testfunFOUR \, \testfun^*
		\,dx
		+
		\Real \int_{\Omega}
		\bigl(
		- 2 \ci \kappa \MagF \cdot \nabla \testfunFOUR
		+ \kappa^2 |\MagF|^2 \testfunFOUR \bigr) \testfun^*
		\,dx ,
	\end{equation}
	and since the second term is in $L^2$, we have, again by regularity theory for the homogenous Neumann problem, $\testfunFOUR\in H^2$ and
	\begin{equation}
		\norm{H^2}{\testfunFOUR} \lesssim  \norm{L^2}{f_\testfunFOUR} + \kappa^2 \norm{L^2}{\testfunFOUR} + \kappa \norm{L^2}{\nabla \testfunFOUR} 
		\lesssim \norm{L^2}{f},
	\end{equation}
	where we used the $L^2$- and $H^1$-bounds for $\testfunFOUR$ in the last step.
\end{proof}

Using the inf-sup stability established in Proposition \ref{proposition-inf-sup-stability}, we can formulate an analogous result for variational problems based on $\energy''(\sol)$ in $\Honeperp$. Note here that the inclusion $\Honeperp \subset H^1$ implies $\dualHone \subset  (\Honeperp)'$ for the dual spaces.

\begin{corollary} \label{cor:du_F_inf_sup_solvability}
	
	Let Assumption~\ref{ass:cinfsup} hold.	
	
	\begin{itemize}
	\item[{\normalfont (a)}] For any $f \in  (\Honeperp)' $, 
	there is a unique $\testfunFOUR \in \Honeperp$ such that
	\begin{equation} \label{eq:du_F_var_problem}
		\dualp{\energy''(\sol) \testfunFOUR}{\testfun}  = \dualp{f}{\testfun}, \quad \text{for all} \quad \testfun \in \Honeperp ,
	\end{equation} 
	which satisfies the estimate
	\begin{equation}
		\Honekappa{\testfunFOUR} \leq \Csol \Honekappaminus{f} .
	\end{equation}
	\item[{\normalfont (b)}] Let $\testfunFOUR \in \Honeperp$ be the solution of \eqref{eq:du_F_var_problem} with $f \in L^2$.
	Then, it further holds
	\begin{align}
		\Honekappa{\testfunFOUR} 
		&\leq \frac{\Csol}{\kappa} \norm{L^2}{f}
		\quad
		\intertext{and, if $\Omega$ is convex, then $\testfunFOUR \in H^2$ and}
		\Htwokappa{\testfunFOUR} &\lesssim \Csol  \norm{L^2}{f} .
	\end{align}
	\end{itemize}
\end{corollary}

\begin{proof}
	By standard theory for indefinite differential equations (cf. \cite{Bab7071}), the inf-sup stability in Proposition~\ref{proposition-inf-sup-stability} directly gives the \wepo of  \eqref{eq:du_F_var_problem} together with the stability estimate $\Honekappa{\testfunFOUR} \leq \Csol \Honekappaminus{f}$, hence proving (a). The first estimate in (b) is obtained from \eqref{eq:relation_Honeminus_L2}.
	Using this observation, we conclude that $\testfunFOUR\in \Honeperp$ solves
	\begin{equation}
		\abilmagstabsym{\testfunFOUR}{\testfun} =  \ipsymLtwo{\widetilde{f}}{\testfun}, \quad  \text{ for all } \testfun \in \Honeperp \,,
	\end{equation}
	for some $\widetilde{f}\in L^2$ with $\norm{L^2}{\widetilde{f}} \lesssim \Csol \norm{L^2}{f}$, 
	and thus Lemma~\ref{lem:wepo_abilmagstab} gives the claim.
\end{proof}

\section{Space discretization and main results}
\label{sec:main}

Let us consider some finite dimensional finite element space $\VSh$ which is a subspace of $H^1(\Omega)$ and
where we recall that we assume $\Omega$ to be a polygonal (or polyhedral) Lipschitz domain. 
By $\h$ we denote a spatial parameter which tends to zero for a finer spatial resolution.

In order to derive a priori estimates for the error between a discrete minimizer $\solh$ and a continuous minimizer $\sol$, we introduce the closed subspace $\VShperp$ of $\VSh$ by
$\VShperp = \VSh \cap (\ci \sol)^\perp \subset \Honeperp$, where the orthogonality is with respect to the inner product $\ipsymLtwo{\cdot}{\cdot}$. Before we proceed, let us stress that $\VShperp$ is just an auxiliary space, that is not required for practical computations but is only used as a tool in the analysis. For the moment, $\sol \in H^1$ can denote any minimizer of the energy $\energy$, however, later we will link every discrete minimizer $\solh \in \VSh$ to a specific minimizer $\sol$ in order to ensure that the error becomes small.

As another theoretical tool, we require the orthogonal projection $\projLtwoisol \colon \Honeperp \rightarrow \VShperp$ to satisfy
\begin{equation}
	\abilmagstabsym{\projLtwoisol w}{\testfunh} = 	\abilmagstabsym{w}{\testfunh}, \quad \text{ for all } \testfunh \in \VShperp.
\end{equation}
In the following assumption, we introduce some abstract conditions which are sufficient to carry out our error analysis and which are later verified for our examples.

\begin{assumption}  \label{ass:FEM_space}
	The family of (non-empty) finite element spaces $\VSh$ has the following properties:
	\begin{itemize}
\item[{\normalfont (a)}] The  family of spaces $\VSh$ is dense in $H^1(\Omega)$ in the sense that for each $\testfun \in H^1$ we have
	\begin{equation}
		\lim\limits_{\h \to 0} \inf\limits_{\testfunh \in \VSh } \norm{H^1}{\testfun - \testfunh} = 0 .
	\end{equation}
\item[{\normalfont (b)}] Let $w \in \Honeperp$ and $f \in L^2$ such that $\abilmagstabsym{w}{\testfun} = \ipsymLtwo{f}{\testfun}$ for all $\testfun \in \Honeperp$. Then, it holds
		\begin{equation} \label{eq:ass_bound_H1kappa}
			\Honekappa{w - \projLtwoisol w} \lesssim  h \norm{L^2}{f} ,
		\end{equation}
	where the constant is independent of $h$ and $\kappa$.
\end{itemize}
\end{assumption}

The most prominent example that fulfills Assumption \ref{ass:FEM_space} are linear Lagrange finite element spaces which are discussed at the end of this section. Property (a) is obvious and to verify property (b), one first replaces $\norm{L^2}{f}$ by $\Htwokappa{w}$, and uses (for a convex domain $\Omega$) $H^2$-regularity. We give the details below.
Another, non-trivial example of generalized finite elements spaces is presented in Section~\ref{sec:lod}.

Recall that we want to minimize the functional $\energy$ from \eqref{eq:energy_functional_times_kappa2} over $\VSh$, i.e.,
\begin{equation} \label{eq:energy_functional_discrete}
	\energy(\solh) = \inf\limits_{\testfunh \in \VSh }
	\energy(\testfunh),\qquad \energy(\testfunh)  = \frac12 \int_\Omega |\nabla \testfunh + \ci \kappa \MagF \testfunh |^2 + \frac{\kappa^2}{2}  \bigl( 1- |\testfunh|^2 \bigr)^2\,dx .
\end{equation}
Note that since $\VSh$ is finite dimensional, the existence of a minimizer $\solh$
is always guaranteed.
Our first result shows bounds on the discrete minimizer $\solh$ in different norms and the corresponding energy, which are independent of $\h$ and behave in the parameter $\kappa$ the same way as the exact minimizer $\sol$ studied in Theorem~\ref{thm:cont_minimizer}. The proof is postponed to Section~\ref{sec:proofs}.

\begin{lemma} \label{lem:bound_general_uH}
	For all $\h>0$ let $\solh$ be a minimizer of \eqref{eq:energy_functional_discrete}.
	Then
	there hold the bounds
	\begin{equation}
		\energy(\solh) \lesssim   \kappa^2
		\qquad 
		\norm{L^2}{\solh} \lesssim 1,
		\qquad
		\norm{L^2}{\nabla \solh}  \lesssim   \kappa,
		\qquad
		\Honekappa{\solh} \lesssim  \kappa ,
	\end{equation}
	where the hidden constants are independent of $\h$ and $\kappa$.
\end{lemma}

Our main findings are collected in the following theorem. 
We provide error bounds for the discrete minimizers which are explicit in the parameter $\kappa$ and the mesh width $\h$.
In addition, we show that the error behaves as the quasi-best approximation of $\Honeperp$ in $\VShperp$.

\begin{theorem} \label{thm:main}
	Let Assumption~\ref{ass:cinfsup} and \ref{ass:FEM_space} hold, and let $\h \leq \h_0$ be sufficiently small such that in particular $\kappa \Csol \h$	is small. Then, there is neighborhood $\Nbh \subset H^1(\Omega)$ of
	each discrete minimizer $\solh$ of \eqref{eq:energy_functional_discrete} such that there is a unique minimizer $\sol \in \Nbh$ of \eqref{eq:energy_functional_times_kappa2} with 
	\begin{equation}
		\ipsymLtwo{\solh}{\ci \sol} = 0,
	\end{equation}
	and we have the error bounds
	\begin{align}
		\Honekappa{ \sol - \solh}  &\lesssim 
		(1 + \kappa \Csol \h)
		\inf_{\testfunh \in \VShperp}	\Honekappa{\sol- \testfunh }, 
		\\  
		\norm{L^2}{ \sol - \solh}  &\lesssim 
		\h \, \Csol \,
		(1 + \kappa \Csol \h)
		\inf_{\testfunh \in \VShperp}	\Honekappa{\sol- \testfunh }  ,
	\end{align}
	as well as the following estimate on the error in the energy 
	\begin{align}
		0\leq 	\energy(\solh) - \energy(\sol) 
		&\lesssim \Honekappa{\sol -\solh}^2 \,
		\bigl( 
		1
		+
		\kappa^{1/2} \Honekappa{\sol -\solh}
		+
		\kappa \Honekappa{\sol -\solh}^2 \bigr),
	\end{align}
	where the hidden constants are independent of $\kappa$, $\Csol$, and $\h$.
\end{theorem}

The proof is divided in several steps which are outlined in detail in Section~\ref{sec:proofs}.
The first application of the results are Lagrangian finite elements.
In the following, we denote by $\Th$ a conforming family of partitions of the domain $\Omega$ consisting of simplical elements $K$. For the space $\Pone(K)$ of complex-valued polynomials of degree less than or equal to $1$ on $K$, we consider the finite element space
\begin{equation} \label{eq:defVh}
	\VSh \coloneqq \{ \testfunh  \in H^1(\Omega) \mid   \testfunh |_K \in \Pone(K) 
	\text{  for all  } 
	K \in \Th \} .
\end{equation}
We assume that the partition $\Th$ is shape-regular and
the $L^2$-projection,  defined via $	\ipsymLtwo{\Ltwoproj v}{\testfunh} = \ipsymLtwo{ v}{\testfunh}$ for all $\testfunh \in \VSh$, is $H^1$-stable, i.e.,
\begin{equation} \label{eq:H1_stab_L2_proj}
	\norm{H^1}{ \Ltwoproj \testfun} \lesssim \norm{H^1}{\testfun} ,
\end{equation}
with a constant  independent of $\h$. This condition is always fulfilled for quasi-uniform triangulations, but is also valid for certain adaptively refined meshes.
For a detailed discussion on criteria when \eqref{eq:H1_stab_L2_proj} holds, we refer to \cite{BanY14,DieST21}.
In this setting, we obtain convergence rates which are explicit in the parameter $\kappa$ and the mesh width $\h$.

\begin{corollary} \label{cor:main_Lagrange}
	Let the conditions of Theorem~\ref{thm:main} hold, and assume in addition that $\Omega$ is convex.
	For Lagrangian finite elements which satisfy \eqref{eq:H1_stab_L2_proj}, the error bounds in Theorem~\ref{thm:main}
	can be further estimated by
	\begin{align}
		\Honekappa{ \sol - \solh}  \lesssim \kappa^2 \h
		\qquad
		\text{and}
		\qquad
		\norm{L^2(\Omega)}{ \sol - \solh}  \lesssim  \Csol \kappa^2 \h^2 ,
	\end{align}
	where the hidden constants are independent of $\kappa$, $\Csol$, and $\h$.
\end{corollary}

The proof is presented in Section~\ref{sec:proofs}.

\section{Numerical experiments}
\label{sec:num_exp}

Before we present the proof of our main result, we illustrate our theoretical findings with some numerical examples confirming the rates and the $\kappa$-dependence in our error bounds. In the following we will hence compute discrete minimizers
$$
\solh = \underset{\testfunh \in \VSh}{\mbox{arg\hspace{2pt}min}}\hspace{1pt}\energy(\testfunh) 
$$
in $\mathcal{P}_1$-spaces $\VSh$ for different mesh sizes $h$ and different material parameters $\kappa$.
Note that, just like analytical minimizers, discrete minimizers can be at most unique up to gauge transformations, i.e., if $\solh \in \VSh$ is a minimizer, so is $\exp(-\ci \theta) \solh$ for any $\theta \in [-\pi,\pi)$.

\subsection{Implementation}

For the discretization in space with linear Lagrange finite elements, we use the open source Python tool FEniCS \cite[version 2019.2.0]{Fenics}. 
To compute a discrete minimizer, we applied a steepest descent approach
using an implicit Euler method for the $L^2$ gradient flow. 
A direct application yields the following nonlinear iteration
\begin{align}
	\ipsymLtwo{\solh^{n+1}}{ \testfunh } = \ipsymLtwo{\solh^{n}}{ \testfunh } - \tau \, \dualp{\energy'(\solh^{n+1})}{\testfunh} ,
\end{align}
where $\tau>0$ is some parameter. To avoid the solution of nonlinear systems several times, we replace $\energy'(u^{n+1})$ by the linearization
\begin{align}
	\dualp{\energy'(\solh^{n+1})}{\testfunh} \to  \abilmag{\solh^{n+1}}{\testfunh} + \kappa^2 \Real \int_\Omega (|\solh^{n}|^2 - 1) \solh^{n+1} \testfunh^* \,dx,
\end{align}
and thus have to solve the following linear system for $\solh^{n+1} \in \VSh$
\begin{align}
	\ipsymLtwo{\solh^{n+1}}{ \testfunh } 
	+ 
	\tau \, \abilmag{\solh^{n+1}}{\testfunh} 
	+
	\tau \, \kappa^2 \Real \int_\Omega (|\solh^{n}|^2 - 1) \solh^{n+1} \testfunh^* \,dx
	= 
	\ipsymLtwo{\solh^{n}}{ \testfunh }
\end{align}
for all $\testfunh \in \VSh$.
In our experiments, we  set $\Omega = [0,1] \times [0,1] \subset \R^2$,
and use 
on the finest grid
the initial value $\sol_0 = 0.8 + 0.6 \ci$.
For the coarser grids, we project this reference solution and use this as a starting value.
The magnetic potential is chosen as
\begin{equation}
	\MagF(x,y) :=  \sqrt{2} \begin{pmatrix}
		\sin(\pi x) \cos(\pi y) \\ -  \cos(\pi x) \sin(\pi y)
	\end{pmatrix} , 
\end{equation}
and satisfies the assumptions in \eqref{eq:ass_\MagF_for_H2}.
Further, we set $\tau = \kappa^{-2}$, and used the stopping criterion
$ \kappa^{-2} | \energy(\solh^{n+1}) - \energy(\solh^{n})| < \delta $
for a tolerance $\delta = 10^{-9}$. Below this tolerance, we use a Newton method for the equation $\energy^{\prime}(\solh)=0$ with the previous approximation $\solh^{n}$ as starting value and the same stopping criterion, but $\delta = 10^{-12}$.
The code to reproduce the results presented in this paper is available 
at \url{\mycode}.

Let us note that the steepest descent approach as described above can potentially also find local minimizers of $\energy$ in $\VSh$, depending on the choice of the initial value. However, we also note that global minimizers are typically found more easily and that local minimizers (that are not global minimizers) can be identified and discarded by comparing the energy levels. 

Finally, let us also mention that the phase of the minimizer (i.e. which $\exp(- \ci \theta) u_h$ we find for some $\theta \in [-\pi,\pi)$) depends only on the choice of the initial value. Hence, it is important to start the steepest descent method always with the same initial value in order to ensure a reasonable comparison between minimizers on different meshes (since phase differences might otherwise dominate the error).

\subsection{Numerical results}

We first illustrate the convergence in the spatial parameter $\h$ for different values of $\kappa$.
To this end, we computed a reference solution on a finer grid using $\h_{\text{max}} \sim  2.5 \cdot 10^{-3}$. 
In order to numerically verify Assumption~\ref{ass:cinfsup} on the local uniqueness of the computed minimizers, we proceeded as in Remark~\ref{rem:loc_uniq} and computed the $5$ smallest eigenvalues (in absolute value) of $\energy''(u_h)$ and collected them in Table~\ref{tab:eigs_L2}.
We observe that the first eigenvalue is essentially zero up to numerical precision and we verified that it indeed belongs to $\ci \solh$. The second smallest eigenvalue is clearly bounded away from zero, and the computed minimum is hence indeed locally unique in the sense of Definition~\ref{def-local-uniqueness}.

\begin{table}
	\begin{center}
		\begin{tabular}{ |c|c|c|c|c|c| } 
			\hline
			$\kappa$ & 		$\lambda_1 $ &	$\lambda_2 $ &	$\lambda_3 $ &	$\lambda_4 $ &	$\lambda_5 $  \\
			\hline
			$8$ & \Zeile{ $\sim 10^{-12}$ }{ 2.65 }{ 2.65 }{ 7.54 }{ 7.71  } \\ 
			\hline
			$ 10$	& \Zeile{ $\sim 10^{-12}$ }{ 3.37 }{ 3.37 }{ 3.78 }{ 10.70 }\\ 
			\hline
			$ 17$	& \Zeile{ $\sim 10^{-8}$ }{ 0.46 }{ 0.57 }{ 2.63 }{ 2.75 } \\	
			\hline
			$ 24$	& \Zeile{ $\sim 10^{-9}$ }{ 0.51 }{ 2.03 }{ 2.03 }{ 5.52 } \\
			\hline
		\end{tabular}
	\end{center}
	\caption{
		The numerically computed $5$ smallest eigenvalues of $\energy''(\solh)$ with reference solution $\solh$ and $\kappa = 8,10,17,24$.
	}
	\label{tab:eigs_L2}
\end{table}

In order to compare the results for different values of $\kappa$, we divide the error in the $\HonekappaSpace$- and $L^2$-norm by $\kappa^2$
and the energy by $\kappa^4$, see Figure~\ref{fig:convergence}. Here we recall that according to Corollary~\ref{cor:main_Lagrange} we expect the $\HonekappaSpace$-error to convergence with the rate $\kappa^2 \h$, the $L^2$-error with the rate $\kappa^2 \h^2$ and the energy-error with the rate $\kappa^4 \h^2$. 
Indeed, we observe the predicted asymptotic convergence in $\h$ and, in particular, the numerical experiments confirm the $\kappa$-scaling in our error estimates. The plot further indicates that the constants in front of the normalized errors are independent of $\kappa$.
Further, we observe in Figure~\ref{fig:uniformness} the uniform boundedness of $\Honekappa{\solh}$ and $\energy(\solh)$ as predicted in Lemma~\ref{lem:bound_general_uH}.

\begin{figure}
	\centering
	\begin{subfigure}{0.45\textwidth}
		\resizebox{1.0\textwidth}{!}{
\begin{tikzpicture}


\begin{axis}[
legend cell align={left},
legend style={
  fill opacity=0.8,
  draw opacity=1,
  text opacity=1,
  at={(0.03,0.97)},
  anchor=north west,
  draw=white!80!black,
  /tikz/every even column/.append style={column sep=3mm}
},
log basis x={10},
log basis y={10},
tick align=outside,
tick pos=left,
title={\large$\kappa^{-2}$-weighted $H^1_\kappa$-error},
x grid style={white!69.0196078431373!black},
xmin=0.00405395301108757, xmax=0.361336322303478,
xmode=log,
xtick style={color=black},
y grid style={white!69.0196078431373!black},
ymin=0.00034704390290097, ymax=1,
ymode=log,
ytick style={color=black},
legend columns=2,
legend to name=Legendforall,
xlabel style={font=\color{white!15!black}},
xlabel={maximal mesh width $h$},
]
\addplot [semithick, color0, mark=asterisk, mark size=3, mark options={solid}]
table {%
	0.294627825494395 0.187756253788867
	0.158363330156169 0.125713923211823
	0.0791364521303744 0.214270410273684
	0.0397638752183648 0.00998237369550496
	0.0198857667264535 0.00470217384436248
	0.00994323397220818 0.00217541209420253
	0.00497183342869752 0.000984159088477717
};
\addlegendentry[color=legendcolor]{$\kappa = 8$}
\addplot [semithick, color0,densely dotted, mark=*, mark size=1.5, mark options={solid}]
table {%
	0.294627825494395 0.0831159822325672
	0.158363330156169 0.0411075743807961
	0.0791364521303744 0.0202679706790136
	0.0397638752183648 0.0100384985540841
	0.0198857667264535 0.00477496962494991
	0.00994323397220818 0.00221791485142096
	0.00497183342869752 0.00100245370590739
};
\addlegendentry[color=legendcolor]{$\kappa = 8$}
\addplot [semithick, color1, mark=asterisk, mark size=3, mark options={solid}]
table {%
	0.294627825494395 0.144750155628437
	0.158363330156169 0.18080647160855
	0.0791364521303744 0.177474552543771
	0.0397638752183648 0.00794967688500174
	0.0198857667264535 0.00369490893034111
	0.00994323397220818 0.00170777072635862
	0.00497183342869752 0.000772109340186972
};
\addlegendentry[color=legendcolor]{$\kappa = 10$}
\addplot [semithick, color1,densely dotted, mark=*, mark size=1.5, mark options={solid}]
table {%
	0.294627825494395 0.0657696935003186
	0.158363330156169 0.0325203600194851
	0.0791364521303744 0.0158945772318894
	0.0397638752183648 0.00789137182395622
	0.0198857667264535 0.00374857419342597
	0.00994323397220818 0.00174115139519233
	0.00497183342869752 0.000786425992192257
};
\addlegendentry[color=legendcolor]{$\kappa = 10$}
\addplot [semithick, color3, mark=asterisk, mark size=3, mark options={solid}]
table {%
	0.294627825494395 0.0760260430003199
	0.158363330156169 0.0948015789935594
	0.0791364521303744 0.0343121568696336
	0.0397638752183648 0.0153100704480831
	0.0198857667264535 0.00428876579308135
	0.00994323397220818 0.00192010629522429
	0.00497183342869752 0.000862385372670657
};
\addlegendentry[color=legendcolor]{$\kappa = 17$}
\addplot [semithick, color3,densely dotted, mark=*, mark size=1.5, mark options={solid}]
table {%
	0.294627825494395 0.0661527497493184
	0.158363330156169 0.0358686994738739
	0.0791364521303744 0.0178148969931917
	0.0397638752183648 0.00877798173969964
	0.0198857667264535 0.00416970404828101
	0.00994323397220818 0.0019344181028929
	0.00497183342869752 0.000876368869041339
};
\addlegendentry[color=legendcolor]{$\kappa = 17$}
\addplot [semithick, color2, mark=asterisk, mark size=3, mark options={solid}]
table {%
	0.294627825494395 0.061239859105911
	0.158363330156169 0.0656540230510088
	0.0791364521303744 0.0702198156268603
	0.0397638752183648 0.0143268026246104
	0.0198857667264535 0.00491216974360704
	0.00994323397220818 0.00202357986256452
	0.00497183342869752 0.000886690037777962
};
\addlegendentry[color=legendcolor]{$\kappa = 24$}
\addplot [semithick, color2, densely dotted, mark=*, mark size=1.5, mark options={solid}]
table {%
	0.294627825494395 0.0552415441159161
	0.158363330156169 0.0365569845207141
	0.0791364521303744 0.0184206433538859
	0.0397638752183648 0.00897820526546169
	0.0198857667264535 0.0042692029680461
	0.00994323397220818 0.00197581101856826
	0.00497183342869752 0.00089536173286726
};
\addlegendentry[color=legendcolor]{$\kappa = 24$}
\addplot [semithick, gray, dashed]
table {%
	0.294627825494395 0.0525447124214099
	0.158363330156169 0.0282429387896053
	0.0791364521303744 0.0141134059970867
	0.0397638752183648 0.00709159559048354
	0.0198857667264535 0.00354648069023141
	0.00994323397220818 0.00177330287365683
	0.00497183342869752 0.000886690037777962
};
\end{axis}
\end{tikzpicture}}	
	\end{subfigure}
	\hfill
	\begin{subfigure}{0.45\textwidth}
		\resizebox{1.0\textwidth}{!}{
\begin{tikzpicture}


\begin{axis}[
legend cell align={left},
legend style={
  fill opacity=0.8,
  draw opacity=1,
  text opacity=1,
  at={(0.03,-0.57)},
  anchor=south,
  draw=white!80!black
},
legend columns=4,
log basis x={10},
log basis y={10},
tick align=outside,
tick pos=left,
title={\large$\kappa^{-2}$-weighted $L^2$-error},
x grid style={white!69.0196078431373!black},
xmin=0.00405395301108757, xmax=0.361336322303478,
xmode=log,
xtick style={color=black},
y grid style={white!69.0196078431373!black},
ymin=6.32426386908231e-07, ymax=0.04,
ymode=log,
ytick style={color=black},
xlabel style={font=\color{white!15!black}},
xlabel={maximal mesh width $h$},
]
\addplot [semithick, color0, mark=asterisk, mark size=3, mark options={solid}]
table {%
0.294627825494395 0.0231888714447355
0.158363330156169 0.0137849228674858
0.0791364521303744 0.0240993739178745
0.0397638752183648 0.000223299236736931
0.0198857667264535 5.50207302372507e-05
0.00994323397220818 1.30973397159786e-05
0.00497183342869752 2.77866291408143e-06
};
\addplot [semithick, color0,densely dotted, mark=*, mark size=1.5, mark options={solid}]
table {%
0.294627825494395 0.00606812128070931
0.158363330156169 0.00172788602675387
0.0791364521303744 0.000457517125164342
0.0397638752183648 0.00011746553759998
0.0198857667264535 2.92215113672642e-05
0.00994323397220818 7.1175807974837e-06
0.00497183342869752 1.60363163158123e-06
};
\addplot [semithick, color1, mark=asterisk, mark size=3, mark options={solid}]
table {%
0.294627825494395 0.0154229941681979
0.158363330156169 0.0180505195115027
0.0791364521303744 0.0176053546020039
0.0397638752183648 0.000201063832983208
0.0198857667264535 4.10484132247657e-05
0.00994323397220818 9.7404411287472e-06
0.00497183342869752 2.04986368810559e-06
};
\addplot [semithick, color1,densely dotted, mark=*, mark size=1.5, mark options={solid}]
table {%
0.294627825494395 0.00449954428121882
0.158363330156169 0.00132847255510768
0.0791364521303744 0.000356726430366279
0.0397638752183648 9.28804507599887e-05
0.0198857667264535 2.3009445076948e-05
0.00994323397220818 5.60141877440793e-06
0.00497183342869752 1.25963253870105e-06
};
\addplot [semithick, color3, mark=asterisk, mark size=3, mark options={solid}]
table {%
0.294627825494395 0.00428968277412085
0.158363330156169 0.00579159838842
0.0791364521303744 0.00172044540723912
0.0397638752183648 0.000728479947763116
0.0198857667264535 7.35333935629601e-05
0.00994323397220818 1.72432443629933e-05
0.00497183342869752 3.47288019594613e-06
};
\addplot [semithick, color3,densely dotted, mark=*, mark size=1.5, mark options={solid}]
table {%
0.294627825494395 0.00344021915169813
0.158363330156169 0.00132840764653573
0.0791364521303744 0.000384164768667122
0.0397638752183648 0.000101292594089078
0.0198857667264535 2.544959847844e-05
0.00994323397220818 6.20224845877356e-06
0.00497183342869752 1.40115062472201e-06
};
\addplot [semithick, color2, mark=asterisk, mark size=3, mark options={solid}]
table {%
0.294627825494395 0.0026540222071163
0.158363330156169 0.00284427519166906
0.0791364521303744 0.00299171161339595
0.0397638752183648 0.000460641271356989
0.0198857667264535 0.000100616780376882
0.00994323397220818 2.30148002175623e-05
0.00497183342869752 4.57838016336046e-06
};
\addplot [semithick, color2,densely dotted, mark=*, mark size=1.5, mark options={solid}]
table {%
0.294627825494395 0.00219385988130862
0.158363330156169 0.00118689694058739
0.0791364521303744 0.000378377585166436
0.0397638752183648 0.000101386011517368
0.0198857667264535 2.59164458209176e-05
0.00994323397220818 6.32634371573772e-06
0.00497183342869752 1.43110187326343e-06
};
\addplot [semithick, gray, dashed]
table {%
0.294627825494395 0.0160777855662791
0.158363330156169 0.00464502400470668
0.0791364521303744 0.00115993021908055
0.0397638752183648 0.000292857341565228
0.0198857667264535 7.32425392090621e-05
0.00994323397220818 1.83119261753959e-05
0.00497183342869752 4.57838016336046e-06
};
\legend{}
\end{axis}

\end{tikzpicture}}	
	\end{subfigure}
	\hfill
	\begin{subfigure}{0.45\textwidth}
		\resizebox{1.0\textwidth}{!}{
\begin{tikzpicture}


\begin{axis}[
legend cell align={left},
legend style={
  fill opacity=0.8,
  draw opacity=1,
  text opacity=1,
  at={(0.03,0.97)},
  anchor=north west,
  draw=white!80!black
},
log basis x={10},
log basis y={10},
tick align=outside,
tick pos=left,
title={\large $\kappa^{-4}$-weighted energy-error},
x grid style={white!69.0196078431373!black},
ymin=1.5784742274581e-07, ymax=0.0026678922926463,
xmode=log,
xtick style={color=black},
y grid style={white!69.0196078431373!black},
ymin=1.5784742274581e-07, ymax=0.0026678922926463,
ymode=log,
ytick style={color=black},
xlabel style={font=\color{white!15!black}},
xlabel={maximal mesh width $h$},
]
\addplot [semithick, color0, mark=asterisk, mark size=3, mark options={solid}]
table {%
0.294627825494395 0.000720846856706405
0.158363330156169 0.000692078526100039
0.0791364521303744 0.000198378522277987
0.0397638752183648 5.25609989919214e-05
0.0198857667264535 1.30002116290676e-05
0.00994323397220818 3.07384996799847e-06
0.00497183342869752 6.13202606155813e-07
};
\addplot [semithick, color1, mark=asterisk, mark size=3, mark options={solid}]
table {%
0.294627825494395 0.000789323773208926
0.158363330156169 0.000427443046844719
0.0791364521303744 0.000121731180610555
0.0397638752183648 3.24329029372437e-05
0.0198857667264535 8.01308242827967e-06
0.00994323397220818 1.89444946702852e-06
0.00497183342869752 3.77574619056104e-07
};
\addplot [semithick, color3, mark=asterisk, mark size=3, mark options={solid}]
table {%
0.294627825494395 0.000401723696809746
0.158363330156169 0.000239827022458215
0.0791364521303744 0.000137799631373733
0.0397638752183648 3.89472470171429e-05
0.0198857667264535 9.82636061154051e-06
0.00994323397220818 2.33093872775727e-06
0.00497183342869752 4.67368814592243e-07
};
\addplot [semithick, color2, mark=asterisk, mark size=3, mark options={solid}]
table {%
0.294627825494395 0.000235808158442704
0.158363330156169 0.000169127779760668
0.0791364521303744 8.27031925769143e-05
0.0397638752183648 3.89904035490416e-05
0.0198857667264535 1.01609957225683e-05
0.00994323397220818 2.42405356029764e-06
0.00497183342869752 4.88059634649743e-07
};
\addplot [semithick, gray, dashed]
table {%
0.294627825494395 0.00171390707400225
0.158363330156169 0.000495163930863367
0.0791364521303744 0.000123649653096551
0.0397638752183648 3.12188682522839e-05
0.0198857667264535 7.8077236165895e-06
0.00994323397220818 1.95206856573853e-06
0.00497183342869752 4.88059634649743e-07
};
\legend{}
\end{axis}

\end{tikzpicture}}	
	\end{subfigure}
	\hfill
	\begin{subfigure}{0.45\textwidth}
		\vspace{-21mm}
		\hspace{8mm}
 	 \resizebox{0.84\textwidth}{!}{\ref{Legendforall}}
	\end{subfigure}
	\caption{Convergence in the mesh size $\h$ for $\kappa$-weighted errors
		in the $\HonekappaSpace$- and $L^2$-norm and for the energy,
		for $\kappa = 8,10,17,24$. The errors between $\sol$ and $\solh$ in $L^2$ and $\HonekappaSpace$ are scaled by $\kappa^{-2}$ and the error in energy by $\kappa^{-4}$.
		The dotted lines indicate the corresponding errors (in $L^2$ and $\HonekappaSpace$ respectively) between $\sol$ and its 
		best-approximation $\textup{R}_{\kappa,\h}(u)$ in $\VSh$ with respect to 
		$\abilmagstabsym{\cdot}{\cdot}$, cf. \eqref{best-approxi-H1kappa}.
		The dashed lines indicate order $\mathcal{O}(\h)$ in the upper left figure, and 
		order $\mathcal{O}(\h^2)$ in the upper left and bottom right figure.
	}
	\label{fig:convergence}
\end{figure}
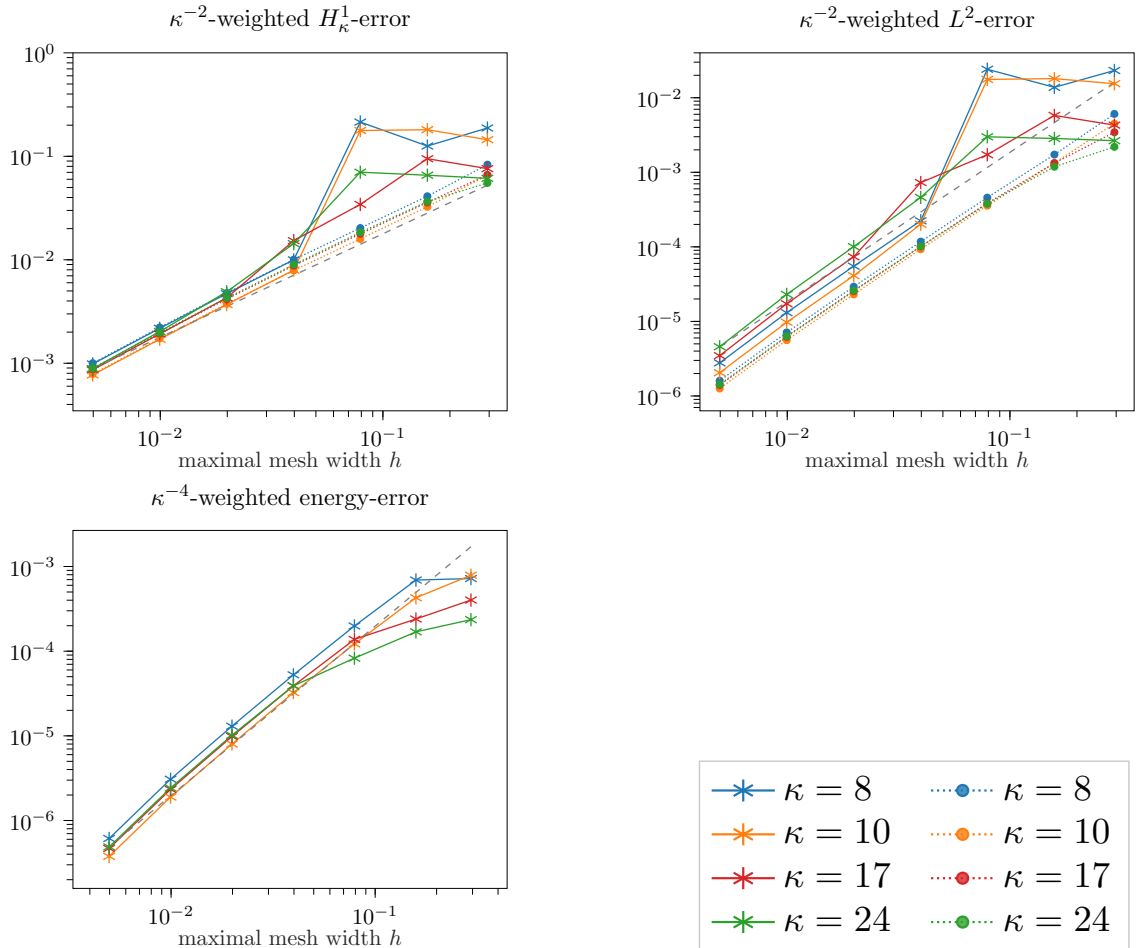

\begin{figure}
	\centering
	\begin{subfigure}{0.45\textwidth}
		\resizebox{1.0\textwidth}{!}{
\begin{tikzpicture}

\begin{axis}[
legend cell align={left},
legend style={
  fill opacity=0.8,
  draw opacity=1,
  text opacity=1,
  at={(0.03,0.97)},
  anchor=north west,
  draw=white!80!black,
  /tikz/every even column/.append style={column sep=3mm}
},
log basis x={10},
log basis y={10},
tick align=outside,
tick pos=left,
title={\large$\kappa^{-1}$-weighted $H^1_\kappa$-norm},
x grid style={white!69.0196078431373!black},
xmin=0.004, xmax=0.5,
xmode=log,
xtick style={color=black},
y grid style={white!69.0196078431373!black},
ymin=0.1, ymax=1.2,
ytick style={color=black},
legend columns=4,
xlabel style={font=\color{white!15!black}},
xlabel={maximal mesh width $h$},
]
\addplot [semithick, color0, mark=asterisk, mark size=3, mark options={solid}]
table {%
0.294627825494395 0.514401168043143
0.158363330156169 0.811869089884772
0.0791364521303744 0.913859595264147
0.0397638752183648 0.939726535166265
0.0198857667264535 0.946574925856512
0.00994323397220818 0.948278890596904
0.00497183342869752 0.948700283533424
};
\addplot [semithick, color1, mark=asterisk, mark size=3, mark options={solid}]
table {%
0.294627825494395 0.417479625304531
0.158363330156169 0.745380738833216
0.0791364521303744 0.824444492260676
0.0397638752183648 0.845729527170419
0.0198857667264535 0.851446858467175
0.00994323397220818 0.852866879728289
0.00497183342869752 0.853217652795658
};
\addplot [semithick, color3, mark=asterisk, mark size=3, mark options={solid}]
table {%
0.294627825494395 0.362387981191403
0.158363330156169 0.56133771645226
0.0791364521303744 0.835631639361613
0.0397638752183648 0.910916062611938
0.0198857667264535 0.932330453552667
0.00994323397220818 0.93767568616111
0.00497183342869752 0.938995366743317
};
\addplot [semithick, color2, mark=asterisk, mark size=3, mark options={solid}]
table {%
0.294627825494395 0.24921514809579
0.158363330156169 0.477769520000272
0.0791364521303744 0.667854785983945
0.0397638752183648 0.907515286403963
0.0198857667264535 0.951484305324516
0.00994323397220818 0.963023480356877
0.00497183342869752 0.965860492400518
};
\legend{}
\end{axis}
\end{tikzpicture}}	
	\end{subfigure}
	\hfill
	\begin{subfigure}{0.45\textwidth}
		\resizebox{1.0\textwidth}{!}{
\begin{tikzpicture}

\definecolor{color0}{rgb}{0.12156862745098,0.466666666666667,0.705882352941177}
\definecolor{color1}{rgb}{1,0.498039215686275,0.0549019607843137}
\definecolor{color2}{rgb}{0.172549019607843,0.627450980392157,0.172549019607843}
\definecolor{color3}{rgb}{0.83921568627451,0.152941176470588,0.156862745098039}
\definecolor{color4}{rgb}{0.580392156862745,0.403921568627451,0.741176470588235}
\definecolor{color5}{rgb}{0.549019607843137,0.337254901960784,0.294117647058824}

\begin{axis}[
legend cell align={left},
legend style={
	fill opacity=0.8,
	draw opacity=1,
	text opacity=1,
	at={(0.03,0.97)},
	anchor=north west,
	draw=white!80!black
},
log basis x={10},
tick align=outside,
tick pos=left,
title={\large$\kappa^{-2}$-weighted energy},
x grid style={white!69.0196078431373!black},
xmin=0.00405395301108757, xmax=0.361336322303478,
xmode=log,
xtick style={color=black},
y grid style={white!69.0196078431373!black},
ymin=0.02, ymax=0.42,
ytick style={color=black},
legend columns=4,
legend to name=Legendforall2,
xlabel style={font=\color{white!15!black}},
xlabel={maximal mesh width $h$},
]
\addplot [semithick, color0, mark=asterisk, mark size=3, mark options={solid}]
table {%
	0.294627825494395 0.174667913751863
	0.158363330156169 0.172826740593055
	0.0791364521303744 0.141229940348444
	0.0397638752183648 0.131897618858136
	0.0198857667264535 0.129365728466913
	0.00994323397220818 0.128730441320605
	0.00497183342869752 0.128572959889447
};
\addlegendentry[color=legendcolor]{ $\kappa = 8$}
\addplot [semithick, color0,densely dotted, mark=*, mark size=1.5, mark options={solid}]
table {%
	0.294627825494395 0.307966533955576
	0.158363330156169 0.181520946333042
	0.0791364521303744 0.142167624331016
	0.0397638752183648 0.132072576409203
	0.0198857667264535 0.129401709133532
	0.00994323397220818 0.128739490297263
	0.00497183342869752 0.12857498871627
};
\addlegendentry[color=legendcolor]{$ \kappa = 8$}
\addplot [semithick, color1, mark=asterisk, mark size=3, mark options={solid}]
table {%
	0.294627825494395 0.183535574986566
	0.158363330156169 0.147347502350145
	0.0791364521303744 0.116776315726729
	0.0397638752183648 0.107846487959398
	0.0198857667264535 0.105404505908502
	0.00994323397220818 0.104792642612376
	0.00497183342869752 0.104640955127579
};
\addlegendentry[color=legendcolor]{$\kappa = 10$}
\addplot [semithick, color1,densely dotted, mark=*, mark size=1.5, mark options={solid}]
table {%
	0.294627825494395 0.278357444397727
	0.158363330156169 0.155981999381869
	0.0791364521303744 0.117693477658881
	0.0397638752183648 0.108018426813211
	0.0198857667264535 0.105439128196237
	0.00994323397220818 0.104801464613601
	0.00497183342869752 0.104642912856309
};
\addlegendentry[color=legendcolor]{$\kappa = 10$}
\addplot [semithick, color3, mark=asterisk, mark size=3, mark options={solid}]
table {%
	0.294627825494395 0.198316077382189
	0.158363330156169 0.151527938494597
	0.0791364521303744 0.122042022471181
	0.0397638752183648 0.0934736833921267
	0.0198857667264535 0.0850577472209076
	0.00994323397220818 0.0828915702964943
	0.00497183342869752 0.0823529985915896
};
\addlegendentry[color=legendcolor]{$\kappa = 17$}
\addplot [semithick, color3,densely dotted, mark=*, mark size=1.5, mark options={solid}]
table {%
	0.294627825494395 0.402158437117071
	0.158363330156169 0.253118243401632
	0.0791364521303744 0.129651541602237
	0.0397638752183648 0.0944299008404986
	0.0198857667264535 0.0852049327704375
	0.00994323397220818 0.0829237392079835
	0.00497183342869752 0.082360029996195
};
\addlegendentry[color=legendcolor]{$\kappa = 17$}
\addplot [semithick, color2, mark=asterisk, mark size=3, mark options={solid}]
table {%
	0.294627825494395 0.203950290466262
	0.158363330156169 0.16554239234541
	0.0791364521303744 0.115761830127568
	0.0397638752183648 0.090583263647513
	0.0198857667264535 0.0739775247394644
	0.00994323397220818 0.0695210460539965
	0.00497183342869752 0.0684059135528233
};
\addlegendentry[color=legendcolor]{$\kappa = 24$}
\addplot [semithick, color2,densely dotted, mark=*, mark size=1.5, mark options={solid}]
table {%
	0.294627825494395 0.370387964660919
	0.158363330156169 0.364801903252744
	0.0791364521303744 0.167934721210636
	0.0397638752183648 0.0936497756663386
	0.0198857667264535 0.0743624619664625
	0.00994323397220818 0.0695916363316121
	0.00497183342869752 0.0684206422077068
};
\addlegendentry[color=legendcolor]{$\kappa = 24$}
\end{axis}

\end{tikzpicture}}	
	\end{subfigure}
	\ref{Legendforall2}
	\caption{
		Boundedness of the scaled $\HonekappaSpace$-norm and energy $\energy$ with respect to $\h$ for $\kappa = 8,10,17,24$.
		The dotted lines show the energy of the best-approximation $\textup{R}_{\kappa,\h}(u)$ in $\VSh$
		with respect to $\abilmagstabsym{\cdot}{\cdot}$.
	}
	\label{fig:uniformness}
\end{figure}
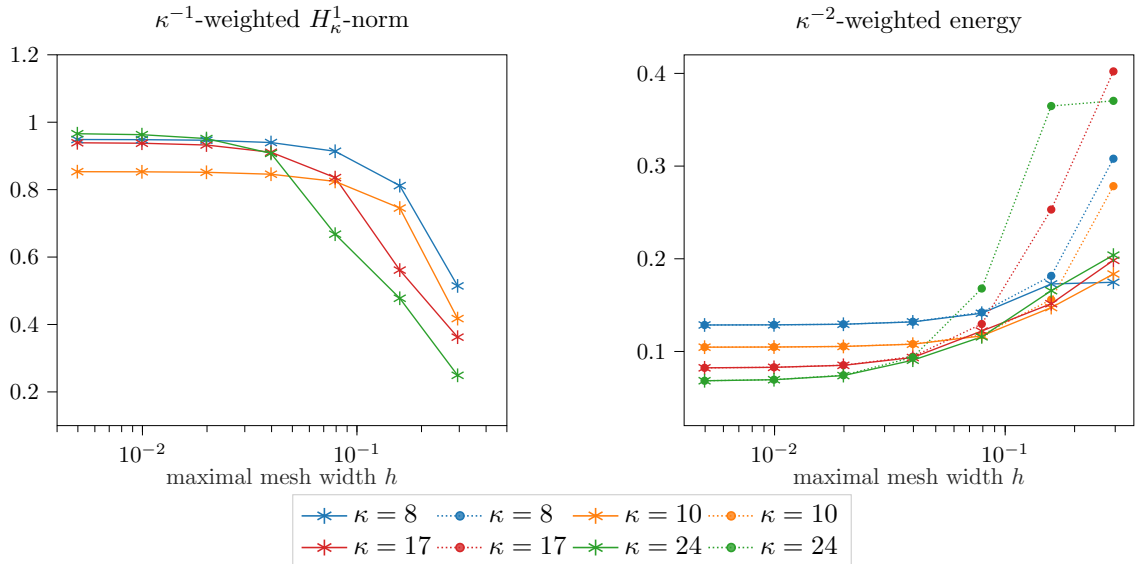

Let us also note that for larger values of $\kappa$, we observe a preasymptotic behavior in $\h$.
We expect that this is related to the smallness condition for $\kappa \Csol \h$ stated in the theorem, which is required below in Lemma~\ref{lem:du_F_inf_sup_solvability_discrete} for the discrete inf-sup stability. Since beyond the (numerically observed) threshold $\kappa h < 1$, the errors coincide for all values of $\kappa$, this is still in alignment with our theory.

To further investigate the preasymptotic effect, we added a comparison of our errors with the 
best-approximation $\textup{R}_{\kappa,\h}(u)$ in $\VSh$ with respect to 
$\abilmagstabsym{\cdot}{\cdot}$.
To be precise, for our reference solution $\sol$, we computed the orthogonal projection $\textup{R}_{\kappa,\h}(\sol)~\in~\VSh$ such that
\begin{align}
	\label{best-approxi-H1kappa}
	\abilmagstabsym{\sol  - \textup{R}_{\kappa,\h}(\sol)}{\phi_h} = 0, \quad \text{for all} \quad \phi_h \in \VSh .
\end{align}
By comparing our errors $\Honekappa{ u - u_h}$ with $\Honekappa{ u - \textup{R}_{\kappa,\h}(u) }$, we can identify possible numerical pollution effects related to $\kappa$. The corresponding results are depicted in Figure \ref{fig:convergence} and reveal a very interesting phenomenon. Contrary to $\solh$ itself, the best-approximation does not show any visible preasymptotic convergence. On the contrary, in the asymptotic regime both approximations have a very similar behavior and follow the expected rates closely. Since both errors are expected to behave asymptotically as
\begin{align}
	\Honekappa{ \sol - \solh} + \Honekappa{ \sol - \textup{R}_{\kappa,\h}(\sol) } \lesssim \kappa \min \{ 1 , \h \kappa \} 
\end{align}
(due to the energy bounds $\Honekappa{\sol}+\Honekappa{\solh}\lesssim \kappa$), we suspect that the necessary resolution condition $\h \kappa \lesssim 1$ will only become visible for larger values of $\kappa$, whereas the preasymptotic convergence of $\solh$ must be related to the smallness condition $\kappa \Csol \h$. This is further supported by the right-hand side of Figure~\ref{fig:uniformness}, where we compare the energy $\energy(\,\textup{R}_{\kappa,\h}(u) \,)$ of the best-approximation with the energy $\energy(\solh)$ of the actual minimizer in $\VSh$. 
The energies in Figure~\ref{fig:uniformness} on the finest level are given by
\begin{align}
	\kappa^{-2} \energy(\solh) &=  1.29 \cdot 10^{-1}  \text{ for } \kappa = 8, \qquad
	& \kappa^{-2} \energy(\solh) &= 8.24 \cdot 10^{-2}   \text{ for } \kappa = 17, 
	\\
	\kappa^{-2}\energy(\solh) &=  1.05\cdot 10^{-1}  \text{ for } \kappa = 10,  \qquad
	& \kappa^{-2} \energy(\solh) &=  6.84\cdot 10^{-2}  \text{ for } \kappa = 24 .
\end{align}
We can see that, in the preasymptotic regime, the energy of $\textup{R}_{\kappa,\h}(u)$ is indeed significantly larger than the one of $\solh$.
This explains why good approximations are not found on coarse meshes and indicates that one cannot capture the correct vortex pattern on meshes which do not satisfy some further resolution condition on $\h$ with respect to $\kappa$.

In our second experiment, we first computed for $\kappa = 20$
the discrete minimizers for different values of 
$h \approx 8\cdot 10^{-2}, 4\cdot 10^{-2}, 2\cdot 10^{-2}, 1 \cdot 10^{-2} $ ,
see Figure~\ref{fig:diff_solutions_in_h}.
We observe that the number of vortices remains constant on the different discretization levels, but 
the minimizer is rotated by $\frac{\pi}{2}$. A simple calculation show that by our choice of $\MagF$ this rotation of the coordinate system leaves the energy invariant. In particular, this illustrates that the density $|u|^2$ of minimizers is not necessarily unique and that convergence of discrete minimizers can only be expected up to a subsequence, even for a fixed gauge condition.
On the other hand, we plotted the minimizers for the values $\kappa=8,10,17,24$, see Figure~\ref{fig:diff_solutions}.
We observe that the number of vortices increases with larger values of $\kappa$, which is in agreement with analytical results \cite{Aftalion99,SaS07}. 

\begin{figure}[t!]
	\begin{subfigure}{0.2\textwidth}
		\includegraphics[width=1.4\textwidth]{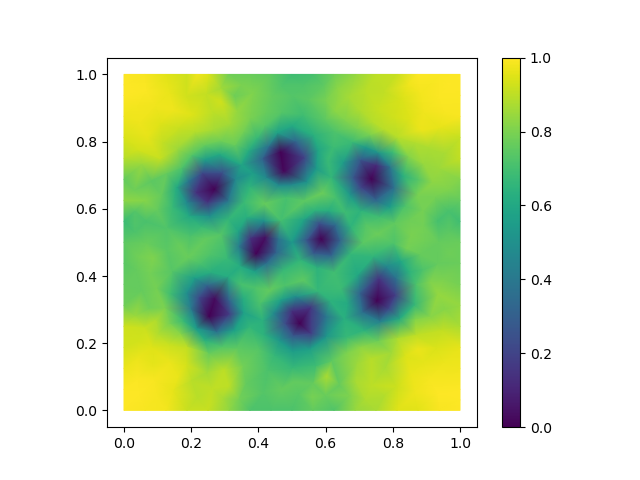}		
	\end{subfigure}
	\hfil
	\begin{subfigure}{0.2\textwidth}
		\includegraphics[width=1.4\textwidth]{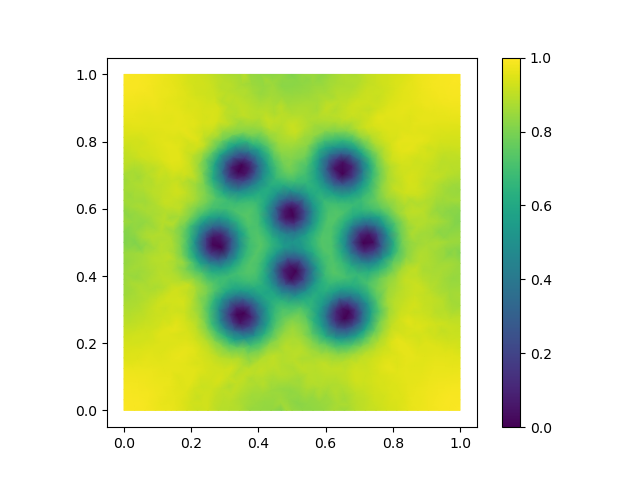}	
	\end{subfigure}
	\hfil
	\begin{subfigure}{0.2\textwidth}
		\includegraphics[width=1.4\textwidth]{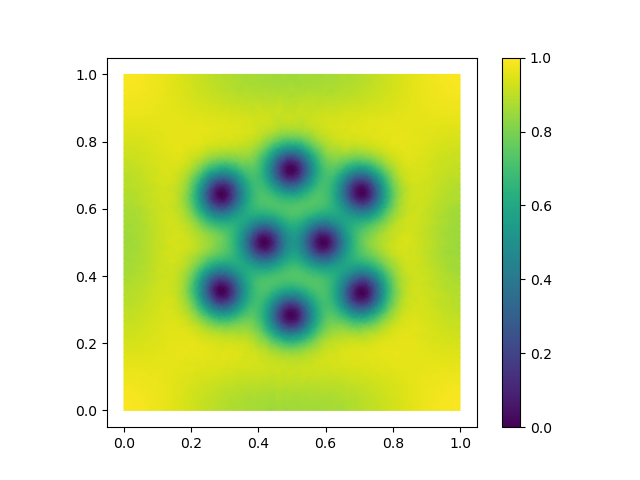}	
	\end{subfigure}
	\hfil
	\begin{subfigure}{0.2\textwidth}
		\includegraphics[width=1.4\textwidth]{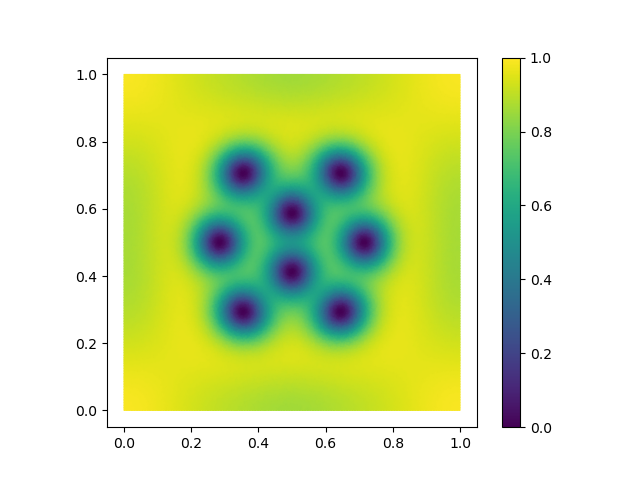}
	\end{subfigure}
	
	\caption{Minimizers for the Ginzburg--Landau parameter $\kappa =20$ and different mesh widths 
		$h \approx 8\cdot 10^{-2}, 4\cdot 10^{-2}, 2\cdot 10^{-2}, 1 \cdot 10^{-2} $ 
		(from left to right).}
	\label{fig:diff_solutions_in_h}
\end{figure}

\begin{figure}[t!]
	\begin{subfigure}{0.2\textwidth}
	\includegraphics[width=1.4\textwidth]{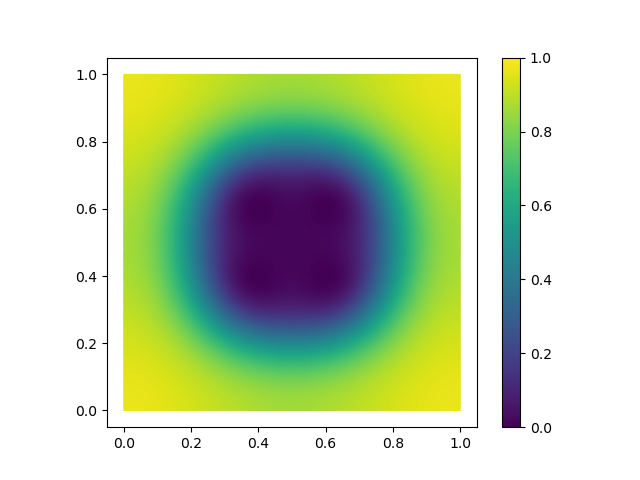}	
	\end{subfigure}
	\hfil
	\begin{subfigure}{0.2\textwidth}
		\includegraphics[width=1.4\textwidth]{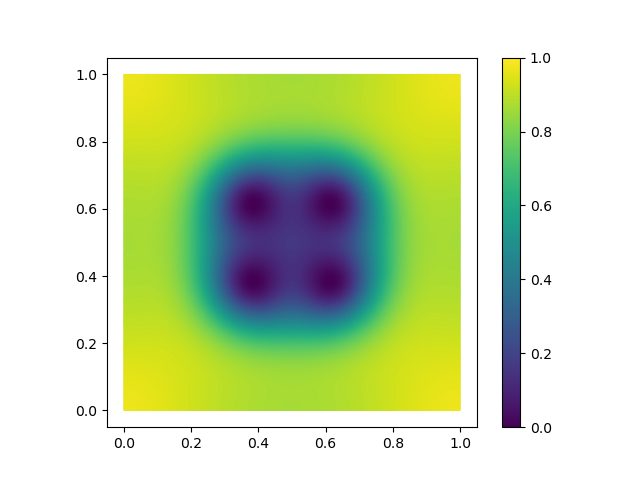}
	\end{subfigure}
	\hfil
	\begin{subfigure}{0.2\textwidth}
		\includegraphics[width=1.4\textwidth]{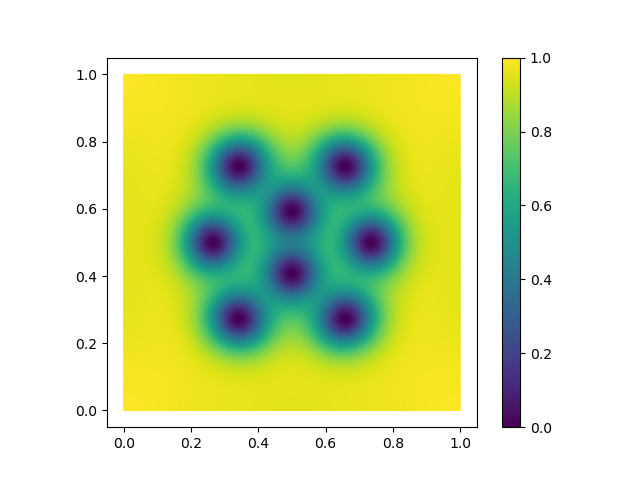}	
		\end{subfigure}
		\hfil
	\begin{subfigure}{0.2\textwidth}
		\includegraphics[width=1.4\textwidth]{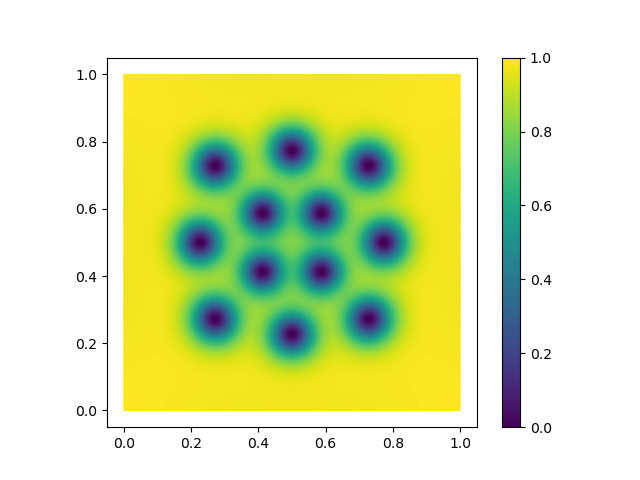}
	\end{subfigure}
	\caption{Different minimizers corresponding to the Ginzburg--Landau parameters $\kappa = 8,10,17,24$ (from left to right) for $\h \approx 2.5 \cdot 10^{-3}$.
	}
	\label{fig:diff_solutions}
\end{figure}

\section{Proof of the main result}
\label{sec:proofs}

In this section, we provide the proof of our main results 
Theorem~\ref{thm:main} and Corollary \ref{cor:main_Lagrange}. We first show an abstract convergence result in order to identify possible limits of a sequence of discrete minimizers. Those are then used to establish convergence with rates, if we are sufficiently close to a continuous minimizer. 	
Throughout this section, we let Assumptions~\ref{ass:cinfsup} and \ref{ass:FEM_space} hold.

\subsection{Abstract convergence result}

In order to deduce convergence, we first establish bounds on minimizers in the discrete space $\VSh$ which are independent of the spatial parameter $\h$ as formulated in Lemma~\ref{lem:bound_general_uH}.

\begin{proof}[Proof of Lemma~\ref{lem:bound_general_uH}]
	First note that for all $\h>0$ we have $0 \in \VSh$, and thus by the minimizing property, we conclude the bound on the energy
	\begin{equation}
		\energy(\solh) \leq \energy(0) 
		\leq \frac{\kappa^2 }{2} \vol(\Omega).
	\end{equation}
	This gives on the one hand
	\begin{equation}
		\norm{L^2}{\nabla \solh + \ci \kappa \MagF \solh } \leq \energy(\solh)^{1/2} \leq \kappa  \vol(\Omega)^{1/2},
	\end{equation}
	and on the other hand we estimate
	\begin{align}
		\frac{\kappa^2}{2} \norm{L^2}{1 - |\solh| \, }^2 \leq 
		\frac{\kappa^2}{2} \int_\Omega   \bigl( 1- |\solh| \bigr)^2 \bigl( 1+  |\solh| \bigr)^2 \,dx \leq  
		\energy (0) = \frac{\kappa^2}{2} \vol(\Omega)^{1/2}  ,
	\end{align}
	and thus conclude
	\begin{equation}
		\norm{L^2}{ \solh} \leq \norm{L^2}{1 - |\solh| } + \vol(\Omega)^{1/2} \leq 2 \vol(\Omega)^{1/2}.
	\end{equation}
	Combining the estimates above, the bound on $\norm{L^2}{\nabla \solh}$ directly follows.
\end{proof}

With the uniform estimates on the discrete minimizers, 
following the approach in \cite{CheGZ10},
we employ the Banach--Alaoglu theorem to obtain some limit
which is an exact minimizer and by Assumption~\ref{ass:cinfsup} locally unique up to complex rotation.

\begin{proposition} \label{prop:abstract_convergence}
	Denote by $(\solh)_{\h>0}$ a family of minimizers of \eqref{eq:energy_functional_discrete}.
	Then, there exists a minimizer $\sol_0$
	of \eqref{eq:energy_functional_times_kappa2} such that there is a monotonically decreasing sequence 
	$( \h_{n} )_{n \in \mathbb{N}}$ with
	\begin{equation}
		\lim_{n\to \infty}
		\Honekappa{ \sol_0 - \solhn{n} } = 0. 
	\end{equation}
	In particular, we can define the twisted approximations
	\begin{equation}
		\solhnTwist{n} \coloneqq e^{\ci \phi_n} \solhn{n}
		\qquad
		\mbox{where } \phi_n \in [-\tfrac{\pi}{2},\tfrac{\pi}{2}] \mbox{ is chosen such that }
		\ipsymLtwo{ \solhnTwist{n} }{\ci \sol_0} = 0 \,,
	\end{equation}
	which 
	also converge in $H^1$, i.e.,  
	\begin{equation}
		\lim_{n\to \infty}
		\Honekappa{ \solhnTwist{n}  - \sol_0} = 0.
	\end{equation}
	Conversely, for any $n$, the minimizer $\solhn{n}$ is an approximation to $e^{ - \ci \phi_n} \sol_0$.
\end{proposition}

\begin{remark}
	The assertion of Proposition~\ref{prop:abstract_convergence} can be interpreted as follows. Assume that there exists a (sub-)sequence of discrete minimizers that keeps a positive distance to all
	exact minimizers, then this would be a contradiction to Proposition~\ref{prop:abstract_convergence}.
	Hence, for $\h$ sufficiently small, one always arrives at a neighborhood of some minimizer $\sol_0$, which is precisely the claim in Theorem~\ref{thm:main}.
\end{remark}

\begin{proof}[Proof of Proposition~\ref{prop:abstract_convergence}]
	The proof of convergence of a subsequence is along the lines of \cite{CheGZ10} if one takes into account the bounds provided in Lemma~\ref{lem:bound_general_uH} together with the weak lower semi-continuity of $\energy$, see Theorem~\ref{thm:cont_minimizer},
	and Assumption~\ref{ass:FEM_space}.

	For the twisted approximations, we note that we can find some $\phi_n \in [-\frac{\pi}{2},\frac{\pi}{2}]$ such that real part of the inner product with $\ci \sol_0$ vanishes if $n$ is large enough.
	Thus, we obtain by the choice of $\phi_n$
	\begin{align}
		\sin \phi_n  \ \ipsymLtwo{ \solhn{n}    }{ \solhn{n}   } 
		&=
		\ipsymLtwo{ e^{\ci \phi_n} \solhn{n}    }{  \ci \solhn{n}   }  
		=
		\ipsymLtwo{ e^{\ci \phi_n} \solhn{n}    }{  \ci \solhn{n} - \ci \sol_0  }  .
	\end{align}
	Since the right-hand side tends to zero, either $\sol_0 = 0$ or
	$\phi_n \to 0$ holds.
	In any case, we have
	\begin{equation}
		\Honekappa{  \solhnTwist{n} - \sol_0} \leq 	\Honekappa{\solhn{n}  - \sol_0} + 	
		|1 - e^{ \ci \phi_n} | \,
		\Honekappa{ \solhn{n}} \to 0 ,
	\end{equation}
	which yields the assertion.
\end{proof}

\subsection{Discrete inf-sup stability}

In order to derive the error estimates, we first establish a discrete version of the inf-sup condition in \eqref{eq:inf_sup_condition_cont}. 
In the proof, we need the following consequence of Assumption~\ref{ass:FEM_space}.

\begin{corollary}
	Let Assumption~\ref{ass:FEM_space} hold, and let $\testfunFOUR \in \Honeperp$ be the solution of
	\begin{equation}
		\dualp{\energy''(\sol) \testfunFOUR}{\testfun}%
		=
		\dualp{f}{\testfun} ,
		\quad \text{ for all } \testfun\in \Honeperp.
	\end{equation}
	Then, it holds the estimate
	\begin{equation} \label{eq:Honekappa_bound_alternative}
		\Honekappa{\testfunFOUR - \projLtwoisol \testfunFOUR} \lesssim   \Csol \h \norm{L^2}{f} .
	\end{equation}
\end{corollary}

\begin{proof}
	According to the representation of $\energy''$ in Lemma~\ref{lem:Frechet_functional}, we bring the terms depending on $\sol$ to the right-hand side
	and for
	$f_\testfunFOUR = \kappa^2 \bigl(  ( |\sol|^2 -1  ) \testfunFOUR + \sol^2 \testfunFOUR^* + |\sol|^2 \testfunFOUR \bigr) -
	\beta  \testfunFOUR $
	we obtain
	\begin{equation}
		\abilmagstabsym{\testfunFOUR}{\testfun} = \ipsymLtwo{f}{\testfun} - \ipsymLtwo{f_\testfunFOUR}{\testfun}  	\quad \text{ for all } \testfun\in \Honeperp.
	\end{equation}
	Here, $f_\testfunFOUR$ satisfies
	\begin{equation}
		\norm{L^2}{f_\testfunFOUR} \lesssim \kappa^2 \norm{L^2}{z} \lesssim \Csol \norm{L^2}{f},
	\end{equation}
	where we used part (b) in Corollary~\ref{cor:du_F_inf_sup_solvability}, and
	the approximation \eqref{eq:ass_bound_H1kappa} in Assumption~\ref{ass:FEM_space}
	gives the claim.
\end{proof}

The proof of the next lemma, which states the discrete inf-sup stability,
is inspired by the thesis \cite[Prop.~8.2.7]{Melenk_Diss95}, where this was done for the Helmholtz equation.

\begin{lemma} \label{lem:du_F_inf_sup_solvability_discrete}
	Let Assumption \ref{ass:cinfsup} be fulfilled.
	\begin{itemize}
	\item[{\normalfont (a)}] If $\kappa \Csol \h$ is sufficiently small, it holds
	\noindent
	for all $\testfunTHREE_\h \in \VShperp$
	\begin{equation}
		\Honekappa{\testfunTHREE_\h} \lesssim \, \Csol  \sup_{\testfunh \in \VShperp} \frac{\dualp{\energy''(\sol)  \testfunTHREE_\h}{\testfunh} } {\Honekappa{\testfunh}},
	\end{equation}
	where the constant is  independent of $\h$ and $\kappa$.
        \item[{\normalfont (a)}] For any $f \in 	(\Honeperp)'$, 
	there is a unique $\testfunTHREE_\h \in \VShperp$ such that
	\begin{equation}
		\dualp{\energy''(\sol)  \testfunTHREE_\h}{\testfunh}%
		=
		\dualp{f}{\testfunh} ,
		\quad \text{ for all } \testfunh \in \VShperp
	\end{equation}
	and it holds
	\begin{equation}
		\Honekappa{\testfunTHREE_\h} \lesssim \Csol \Honekappaminus{f} .
	\end{equation}
	\end{itemize}
\end{lemma}

\begin{proof}
	Part \bulletpoint{b} is a classical stability bound for inf-sup stable problems, cf. \cite[Thm.~2.1]{Bab7071}. Hence, claim \bulletpoint{b} directly follows once we have shown \bulletpoint{a}. To do so,
	we fix $\testfunTHREE_\h \in \VShperp$ and observe for arbitrary $\testfunFOUR \in \Honeperp$
	\begin{equation} 	\label{proof:lem5.3:step1}
	\begin{aligned}
		&\dualp{\energy''(\sol)  \testfunTHREE_\h}{\testfunTHREE_\h + \testfunFOUR } 
		\\
		&\quad = \abilmag{\testfunTHREE_\h}{\testfunTHREE_\h}
		+
		\kappa^2 \Real \int_\Omega \bigl( 2 |\sol|^2 -1 \bigr)  \testfunTHREE_\h \testfunTHREE_\h^*  + u^2 \testfunTHREE_\h^* \testfunTHREE_\h^*  
		\,dx
		+
		\dualp{\energy''(\sol)  \testfunFOUR}{\testfunTHREE_\h}  .
	\end{aligned}
	\end{equation}
	Let $\testfunFOUR \in \Honeperp$ be the unique solution to
	\begin{align}
		\dualp{\energy''(\sol)  \testfunFOUR}{\testfun}   = 	\ipsymLtwo{f}{\testfun} \quad \text{ for all } \testfun \in \Honeperp
		\qquad \text{with} \quad  f = ( \stabPar^2 + 2 \kappa^2 ) \testfunTHREE_\h, 
	\end{align}
	which exists by Proposition~\ref{proposition-inf-sup-stability} and Corollary \ref{cor:du_F_inf_sup_solvability}, and insert it into \eqref{proof:lem5.3:step1}.
	Then, we obtain from \eqref{eq:def_abilmagstabsym} together with Lemma~\ref{lem:prop_bil} and Lemma~\ref{lem:Frechet_functional} that
	\begin{align}
		\Honekappa{\testfunTHREE_\h}^2 &\lesssim \dualp{\energy''(\sol)  \testfunTHREE_\h}{\testfunTHREE_\h + \testfunFOUR }
		\lesssim
		\dualp{\energy''(\sol)  \testfunTHREE_\h}{\testfunTHREE_\h +  \projLtwoisol \testfunFOUR }
		+ 
		\Honekappa{\testfunTHREE_\h} \Honekappa{ \projLtwoisol \testfunFOUR - \testfunFOUR }
		\\
		&\lesssim 
		\sup_{\testfunh \in \VShperp} \frac{\dualp{\energy''(\sol)  \testfunTHREE_\h}{\testfunh} } {\Honekappa{\testfunh}} 
		\Honekappa{\testfunTHREE_\h + \projLtwoisol \testfunFOUR}
		+
		\Honekappa{\testfunTHREE_\h} \Honekappa{ \projLtwoisol \testfunFOUR - \testfunFOUR }.
	\end{align}
	It remains to study the terms with $\testfunFOUR$. Here, 
	we establish with Corollary~\ref{cor:du_F_inf_sup_solvability}
	and \eqref{eq:Honekappa_bound_alternative} the bound
	\begin{equation}
		\kappa 	\Honekappa{\testfunFOUR} 
		+
		\h^{-1} 	\Honekappa{\testfunFOUR - \projLtwoisol \testfunFOUR} 
		\lesssim \Csol \norm{L^2}{f} \lesssim \kappa \Csol \Honekappa{\testfunTHREE_\h}.
	\end{equation}
	From this, we finally conclude 
	\begin{align}
		\Honekappa{\testfunTHREE_\h}^2 &\lesssim 
		\sup_{\testfunh \in \VShperp} \frac{\dualp{\energy''(\sol)  \testfunTHREE_\h}{\testfunh} } {\Honekappa{\testfunh}} 
		\Csol \Honekappa{\testfunTHREE_\h}
		+
		\kappa \Csol \h \Honekappa{\testfunTHREE_\h}^2 ,
	\end{align}
	and obtain the assertion \bulletpoint{a} if $\kappa \Csol \h$ is sufficiently small by absorption.
\end{proof}

\subsection{Convergence with rates}

After these preparations, we can derive the error equation employing the second \Frechet derivative $\energy''$. To this end, we strive for a representation of the form
\begin{align} \label{eq:error_equation}
	\dualp{\energy''(\sol)  ( \projLtwoisol   \sol -\solh)}{\testfunh}	 &= 
	\errHmone (\testfunh) ,
\end{align}
for $\testfunh \in \VShperp$ and employ Lemma~\ref{lem:du_F_inf_sup_solvability_discrete} to conclude a bound for $\projLtwoisol \sol -\solh$. The right-hand side $\errHmone$ is studied in the following lemma.

\begin{lemma} \label{lem:rep_error}
	
	Let $\sol$ and $\solh$  be minimizers of \eqref{eq:energy_functional_times_kappa2}
	and \eqref{eq:energy_functional_discrete}, respectively. 
	\begin{itemize}
	\item[{\normalfont (a)}] For $\testfunh \in \VShperp$ it holds the representation \eqref{eq:error_equation}
	where $\errHmone = \errHmonelin + \errHmonenonlin$ and
	\begin{align} \label{eq:rep_error_Lagrange}
			\errHmonelin (\testfunh) 
			&=
			\kappa^2 \Real \int_\Omega 
			\Bigl( 
			(2 |\sol|^2 - 1) ( \projLtwoisol   \sol -\sol) + \sol^2 ( \projLtwoisol   \sol -\sol)^*
		  \Bigr) \testfunh^* \,dx
			\\
			&\qquad - \stabPar^2 \Real \int_\Omega 
			( \projLtwoisol   \sol -\sol)  \testfunh^* \,dx ,
			\\
			\errHmonenonlin (\testfunh )
			&= 2 \kappa^2
			\Real  \int_\Omega |\sol|^2 \sol \testfunh^* 	\,dx
			+
			\kappa^2
			\Real  \int_\Omega |\solh|^2 \solh \testfunh^* 	\,dx
			\\
			&\quad
			 -
			\kappa^2
			\Real  \int_\Omega
			2 \bigl(  |\sol|^2 \solh + \sol^2 \solh^* \bigr)  \testfunh^*  \,dx .
		\end{align} 
	 	\item[{\normalfont (b)}] The error terms are bounded by
		\begin{align}
			\Honekappaminus{\errHmonelin}   &\lesssim
			\kappa \norm{L^2}{\sol -  \projLtwoisol   \sol},
			\\
			\Honekappaminus{\errHmonenonlin} &\lesssim \kappa   \bigl( \norm{L^{4}}{\sol - \solh}^2   
			+ \norm{L^{6}}{\sol - \solh}^3   \bigr)  ,
		\end{align}
		where the constants are independent of $\h$ and $\kappa$.
	\end{itemize}
	\end{lemma}

	\begin{proof}
		Inserting the exact solution $\sol$, we decompose $\errHmone$ as
		\begin{align} 
			\errHmonelin (\testfunh)  = \dualp{\energy''(\sol)  ( \projLtwoisol   \sol -\sol)}{\testfunh},
			\qquad
			\errHmonenonlin (\testfunh)  = \dualp{\energy''(\sol)  ( \sol - \solh)}{\testfunh},
		\end{align}
		and treat the two terms separately. We begin with the linear part and use Lemma~\ref{lem:Frechet_functional}, the definition of $\abilmagstabsym{\cdot}{\cdot}$ in \eqref{eq:def_abilmagstabsym}, and the orthogonality condition of $\projLtwoisol$ to obtain
		\begin{align}
 			& \dualp{\energy''(\sol)  ( \projLtwoisol   \sol -\sol)}{\testfunh}
			\\
			&=\abilmag{ \projLtwoisol   \sol -\sol}{\testfunh}
			\\
			&\quad 
			+ \kappa^2 \Real \int_\Omega 
			\Bigl( 
			(|\sol|^2 - 1) ( \projLtwoisol   \sol -\sol) + \sol^2 ( \projLtwoisol   \sol -\sol)^*+ |\sol|^2 ( \projLtwoisol   \sol -\sol)  \Bigr) \testfunh^* \,dx
			\\
			&=- \stabPar^2 \Real \int_\Omega 
			( \projLtwoisol   \sol -\sol)  \testfunh^* \,dx
			\\
			&\quad 
			+  \kappa^2 \Real \int_\Omega 
			\Bigl( 
			(|\sol|^2 - 1) ( \projLtwoisol   \sol -\sol) + \sol^2 ( \projLtwoisol   \sol -\sol)^*+ |\sol|^2 ( \projLtwoisol   \sol -\sol)  \Bigr) \testfunh^* \,dx.
		\end{align}
		Using that $\kappa \norm{L^2}{\testfunh} \leq \Honekappa{\testfunh}$  gives the first estimate in part \bulletpoint{b}. 
		
		For the nonlinear part, we note with Lemma~\ref{lem:Frechet_functional} the identity for $\testfunTWO,\testfun\in H^1$
		\begin{equation}
			\dualp{	\energy''(\testfunTWO) \testfunTWO}{\testfun}= \dualp{	\energy'(\testfunTWO)}{\testfun} + 2 \kappa^2 \Real \int_{\Omega}  |\testfunTWO|^2 \testfunTWO \testfun^*  \,dx .
		\end{equation}
		Since $\dualp{	\energy'(\sol)}{ \testfunh } = \dualp{	\energy'(\solh)}{ \testfunh } = 0 $, we expand 
		\begin{align}
			\dualp{	\energy''(\sol)  (\sol - \solh)}{\testfunh}	
			&=
			\dualp{	\energy''(\sol)  \sol }{ \testfunh } 
			-
			\dualp{	\energy''(\solh)  \solh }{ \testfunh } 
			+
			\dualp{	\bigl( \energy''(\solh) - \energy''(\sol)  \bigr)  \solh }{ \testfunh } 
			\\
			&= 
			2 \kappa^2
			\Real  \int_\Omega |\sol|^2 \sol \testfunh^* 	\,dx
			-
			2 \kappa^2
			\Real  \int_\Omega |\solh|^2 \solh \testfunh^* 	\,dx
			\\
			&\quad + \kappa^2
			\Real  \int_\Omega
			2 \bigl( |\solh|^2 - |\sol|^2 \bigr) \solh \testfunh^*  
			+ \bigl( \solh^2 - \sol^2 \bigr) \solh^* \testfunh^* 
			\,dx
			\\
			&= 
			2 \kappa^2
			\Real  \int_\Omega |\sol|^2 \sol \testfunh^* 	\,dx
			+
			\kappa^2
			\Real  \int_\Omega |\solh|^2 \solh \testfunh^* 	\,dx
			\\
			&\quad -
			\kappa^2
			\Real  \int_\Omega
			\bigl(  2 |\sol|^2 \solh + \sol^2 \solh^* \bigr)  \testfunh^*  \,dx ,
		\end{align}
		where we collected terms in the last step. For the estimate, we write
		$\solh = \sol - \errh$
		and compute
		\begin{align} \label{eq:Taylor_nonlinear_errh}
			\ 2 |\sol|^2 \sol  +  |\solh|^2 \solh  -
			\bigl( 2 |\sol|^2 \solh + \sol^2 \solh^* \bigr) 
			&= 2  \sol |\errh|^2 + \errh^2 \sol^* - |\errh|^2 \errh  ,     
		\end{align}
		which together with $|\sol|\leq 1$ and the H{\"o}lder inequality gives the second bound.
	\end{proof}
	
	Now we have everything together to prove the first part of Theorem~\ref{thm:main}, i.e., the $\HonekappaSpace$-estimates for the discrete minimizers.
	
	\begin{proposition} \label{prop:convergence_ahat_general}
		Let $\sol$ and $\solh$  be minimizers of \eqref{eq:energy_functional_times_kappa2}
		and \eqref{eq:energy_functional_discrete}, respectively, and assume the orthogonality  $\ipsymLtwo{\solh}{\ci \sol} = 0$.
	\begin{itemize}
	\item[{\normalfont (a)}] We have for the fully discrete error
		\begin{align}
			\Honekappa{ \sol - \solh}  &\lesssim  
			\Honekappa{\sol- \projLtwoisol \sol  }  
			+
			\kappa \, \CsolH
			\norm{L^2}{\sol- \projLtwoisol \sol  } 
			\\
			&\quad
			+ 
			\kappa \, \CsolH
			\bigl( \norm{L^4}{\sol - \solh }^2 
			+
			\norm{L^6}{\sol - \solh}^3 \bigr) \,.
		\end{align}
	\item[{\normalfont (b)}] For $\h$ small enough and the (unique) minimizer $\sol$ in  Proposition~\ref{prop:abstract_convergence}, it holds
		\begin{align}
			\Honekappa{ \sol - \solh}  &\lesssim 
			\Honekappa{\sol- \projLtwoisol \sol  }  
			+
			\kappa \, \CsolH
			\norm{L^2}{\sol- \projLtwoisol \sol  } .
		\end{align}
	\end{itemize}		
	\end{proposition}
	
	Let us point out that \bulletpoint{a} holds for any minimizers $\sol$ and $\solh$. But to ensure that the higher order terms are indeed negligible, we need the a priori information from the abstract convergence result in Proposition~\ref{prop:abstract_convergence}.  
	
	\begin{proof}[Proof of Proposition~\ref{prop:convergence_ahat_general}]
		\bulletpoint{a} Using the triangle inequality, we obtain
		\begin{align}
			\Honekappa{ \sol - \solh}  &\lesssim 
			\Honekappa{\sol-  \projLtwoisol   \sol  }  
			+
			\Honekappa{  \projLtwoisol   \sol - \solh},
		\end{align}
		and are left to bound the second term.  Lemmas~\ref{lem:du_F_inf_sup_solvability_discrete}
		and \ref{lem:rep_error} then give
		\begin{align}
			\Honekappa{  \projLtwoisol   \sol - \solh} 
			&\lesssim
			\Csol \Honekappaminus{\errHmone}
			\\
			&\lesssim
			\kappa \Csol  \bigl(   \norm{L^2}{\sol -  \projLtwoisol   \sol}
			+
			\norm{L^{4}}{\sol - \solh}^2   
			+ \norm{L^{6}}{\sol - \solh}^3    \bigr) ,
		\end{align}
		and the bound is established.
		
		\bulletpoint{b} With the convergence shown in Proposition~\ref{prop:abstract_convergence} for $\h$ sufficiently small, we can absorb the higher order terms, and obtain the claimed estimate for $\h \leq \h_0$.
	\end{proof}
	
	We can further show quadratic convergence in the $L^2$-norm for the discrete minimizers using an Aubin--Nitsche argument.

	\begin{lemma} \label{lem:convergence_ahat_general_L2}
		Let $\sol$ and $\solh$  be a minimizers of \eqref{eq:energy_functional_times_kappa2}
		and \eqref{eq:energy_functional_discrete}, respectively, and assume the orthogonality  $\ipsymLtwo{\solh}{\ci \sol} = 0$.
		We have for the fully discrete error
		\begin{align}
			\norm{L^2}{ \sol - \solh } &\lesssim  \Csol   
			\h  \Honekappa{\sol - \solh} \\
			&\qquad + 
			\Csol  \kappa \norm{L^2}{\sol - \solh} \bigl(
			\norm{L^3}{\sol - \solh } 
			+
			\norm{L^6}{\sol - \solh}^2  
			\bigr) ,
		\end{align}
		and hence for $\h$ sufficiently small, it holds
		for the (unique) minimizer $\sol$ in  Proposition~\ref{prop:abstract_convergence}
		\begin{align}
			\norm{L^2}{\sol - \solh} 
			\lesssim 
			\Csol 	
			\h 
			\Honekappa{\sol-  \solh  } .
		\end{align}
	\end{lemma}

	\begin{proof}
		Recall the abbreviation  $\errh = \sol - \solh$, and let $\testfunFOUR \in \Honeperp$ be the solution of 
		\begin{equation}
			\ipsymLtwo{\energy''(\sol)  \testfunFOUR }{\testfun} = \ipsymLtwo{ \errh }{\testfun} ,
		\end{equation}
		and note that Corollary~\ref{cor:du_F_inf_sup_solvability} and
		\eqref{eq:Honekappa_bound_alternative}
		give the estimate
		\begin{equation} \label{eq:H2_Aubin_Nitsche}
			\kappa \Honekappa{\testfunFOUR} + \h^{-1} \Honekappa{\testfunFOUR - \projLtwoisol \testfunFOUR} \leq  \Csol  \norm{L^2}{ \errh }.
		\end{equation}
		Using the symmetry of $\energy''$, we can decompose the error as
		\begin{align}
			\norm{L^2}{ \errh  }^2 &= 
			\dualp{\energy''(\sol) \, \errh }{ \testfunFOUR - \projLtwoisol \testfunFOUR}
			+
			\dualp{\energy''(\sol) \, \errh }{ \projLtwoisol \testfunFOUR }
			%
			= E^1 +	\errHmonenonlin (\projLtwoisol \testfunFOUR ) ,
		\end{align}
		where $\errHmonenonlin$ is defined in Lemma~\ref{lem:rep_error}.
		We estimate the first term with Lemma~\ref{lem:Frechet_functional} 
		and \eqref{eq:H2_Aubin_Nitsche} 
		\begin{align}
			E^1 &\lesssim
			\Honekappa{\errh} \Honekappa{\testfunFOUR - \projLtwoisol \testfunFOUR}
			\lesssim
			\Csol   \h   \Honekappa{e} \norm{L^2}{ \errh } .
		\end{align}
		For the second term, we use the representation of $\errHmonenonlin$ in
		Lemma~\ref{lem:rep_error} and
		\eqref{eq:Taylor_nonlinear_errh} 
		together with \eqref{eq:H2_Aubin_Nitsche} and the H{\"o}lder equation
		to obtain 
		\begin{align}
			| \errHmonenonlin (\projLtwoisol \testfunFOUR ) | &\lesssim
			\kappa^2  \norm{L^{2}}{\sol - \solh}  \bigl( \norm{L^{3}}{\sol - \solh}   
			+ \norm{L^{6}}{\sol - \solh}^2   \bigr)  \norm{H^1}{\projLtwoisol \testfunFOUR}
			\\
			&\lesssim 
			\kappa \Csol    \norm{L^{2}}{\sol - \solh}  
			\bigl( \norm{L^{3}}{\sol - \solh}  	+ \norm{L^{6}}{\sol - \solh}^2   \bigr) 
			\norm{L^2}{\errh} ,
		\end{align}
		where we used $\kappa \norm{H^1}{\projLtwoisol \testfunFOUR} \lesssim \kappa \Honekappa{\testfunFOUR} \lesssim \Csol \norm{L^2}{\errh}  $ 
		in the last step.
		Combining the two bounds and dividing by $\norm{L^2}{\errh}$ gives the desired estimate.
	\end{proof}
	
	A similar trick gives the improved convergence of $\projLtwoisol$ in the $L^2$-norm.
	
	\begin{lemma} \label{lem:proj_L2}
		For $\kappa \h$ small enough, the following bound holds for all $\testfunTHREE \in \Honeperp$
		\begin{equation}
			\norm{L^2}{ \testfunTHREE - \projLtwoisol (\testfunTHREE) } 
			\lesssim \h 	\Honekappa{ \testfunTHREE - \projLtwoisol (\testfunTHREE) } ,
		\end{equation}
		where the constant is independent of $\h$ and $\kappa$.
	\end{lemma}
	
	\begin{proof}
		We use an Aubin--Nitsche argument and let $\testfunFOUR \in
		\Honeperp$ be the solution of
		\begin{equation}
			\abilmagstabsym{\testfunFOUR}{\testfun} = \ipsymLtwo{ \testfunTHREE - \projLtwoisol \testfunTHREE}{\testfun}, \quad \text{ for all } \testfun \in \Honeperp .
		\end{equation} 
		Using orthogonality, we have by \eqref{eq:ass_bound_H1kappa} that
		\begin{equation}
			\norm{L^2}{\testfunTHREE - \projLtwoisol \testfunTHREE}^2
			=
			\abilmagstabsym{\testfunFOUR - \projLtwoisol \testfunFOUR}{\testfunTHREE - \projLtwoisol \testfunTHREE} 
			\lesssim
			\h \norm{L^2}{\testfunTHREE - \projLtwoisol \testfunTHREE}	\Honekappa{\testfunTHREE - \projLtwoisol \testfunTHREE},
		\end{equation}
		and the claim follows.
	\end{proof}

	Finally, we provide the error bounds for the energy which behaves in the lowest order as the square of the error in the $\HonekappaSpace$-norm.
	
	\begin{lemma} \label{lem:convergence_energy}
		Let $\sol$ and $\solh$  be minimizers of \eqref{eq:energy_functional_times_kappa2}
		and \eqref{eq:energy_functional_discrete}, respectively.
		The error in the energies is bounded by
		\begin{equation}
			0\leq 	\energy(\solh) - \energy(\sol) 
			\lesssim
			\Honekappa{\sol -\solh}^2 \,
			\bigl( 
			1
			+
			\kappa^{1/2} \Honekappa{\sol -\solh}
			+
			\kappa \Honekappa{\sol -\solh}^2 \bigr).
		\end{equation}
	\end{lemma}
	
	We note that the powers of $\kappa$ can be improved in the case $d =2$, but since the leading order term does not change, we will not give any details here.
	
	\begin{proof}[Proof of Lemma~\ref{lem:convergence_energy}]
		Since $\VSh \subset \VS$, we have $\energy(\sol) \leq \energy(\solh)$, and thus the lower bound. In the next step, we derive the representation
		\begin{equation}\label{eq:rep_energy}
			\begin{aligned} 
				\energy(\solh) - \energy(\sol) 
				&=
				\frac12 \abilmag{\sol - \solh}{\sol - \solh}
				\\
				&\qquad +
				\frac{\kappa^2}{4} \Real \int_\Omega
				(1-|\solh|^2)^2 -	(1-|\sol|^2)^2
				+
				4 (|\sol|^2 -1 ) \sol (\sol - \solh)^* \,dx¸
			\end{aligned}
		\end{equation}
		Let us first note the identity
		\begin{equation}
			\frac12 \abilmag{\sol - \solh}{\sol - \solh} = 
			\frac12 \abilmag{\solh}{\solh} 
			-
			\frac12 \abilmag{\sol}{\sol} 
			+
			\abilmag{\sol}{\sol - \solh} ,
		\end{equation}
		and rewrite the energies as 
		\begin{align}
			& \, \energy(\solh) - \energy(\sol) 
			\\
			 &= \frac12 \abilmag{\solh}{\solh} 
			-
			\frac12 \abilmag{\sol}{\sol} 
			+
			\frac{\kappa^2}{4}
			\Real \int_\Omega
			(1-|\solh|^2)^2 -	(1-|\sol|^2)^2 \,dx
			\\
			&= \frac12 \abilmag{\sol - \solh}{\sol - \solh} 
			+
			\frac{\kappa^2}{4}
			\Real \int_\Omega
			(1-|\solh|^2)^2 -	(1-|\sol|^2)^2 \,dx
			-
			\abilmag{\sol}{\sol - \solh}.
		\end{align}
		Since $\sol$ is a minimizer, we have $\dualp{\energy'(\sol)}{\sol -\solh} = 0$
		and thus by \eqref{eq:first_Frechet}
		\begin{equation}
			- \abilmag{\sol}{\sol - \solh} =  \kappa^2 \Real \int_\Omega
			(|\sol|^2 -1 ) \sol (\sol - \solh)^* \,dx ,
		\end{equation}
		and hence \eqref{eq:rep_energy} holds. The first term of the representation gives the $\HonekappaSpace$-norm in the estimate,
		and it remains to study  the nonlinear part. We first investigate the difference of the squares.
		As before, we write 
		$\solh = \sol - \errh$
		and obtain
		\begin{align}
			\bigl(1-|\sol - \errh|^2\bigr)^2  &= 
			\bigl( |\sol|^2 + |\errh|^2 - 1 - 2 \Real (\sol \errh^*)  \bigr)^2 
			\\
			&= |\sol|^4 + 1 - 2 |\sol|^2
			 +
			  4 (1-|\sol|^2 ) \Real (\sol \errh^*) 
			+ \mathcal{O}(|\errh|^2 + |\errh|^3 + |\errh|^4) ,
		\end{align}
		which gives
		\begin{align}
			(1-|\solh|^2)^2 -	(1-|\sol|^2)^2 
			=
			 4  (1-|\sol|^2 )  \Real (\sol \errh^*) 
			+ \mathcal{O}(|\errh|^2 + |\errh|^3 + |\errh|^4) .
		\end{align}
		We now show that the part, which is linear in $\errh$, is
		canceled by the last term in \eqref{eq:rep_energy}.
		In fact, since it holds
		\begin{equation}
			4 \Real (|\sol|^2 -1 ) \sol (\sol - \solh)^*  =  
			4 |\sol|^2  \Real (\sol \errh^* ) 
			- 4 \Real (\sol \errh^*) ,
		\end{equation}
		we conclude from \eqref{eq:rep_energy}, the fact that $|\sol|\leq 1$ and the H{\"o}lder inequality the bound 
		\begin{equation}
			\energy(\solh) - \energy(\sol) 
			\lesssim \Honekappa{\sol -\solh}^2
			+
			\kappa^2
			\bigl( 
			\norm{L^2}{\sol-\solh}^2 +  \norm{L^3}{\sol-\solh}^3
			+
			\norm{L^4}{\sol-\solh}^4 \bigr).
		\end{equation}
		To show the final estimate, we use interpolation theory, see e.g., \cite[Thm.~2.6]{Lun18}, 
		with $\tfrac{1}{3} = \tfrac{\theta}{2} + \tfrac{1-\theta}{6}$ for $\theta = \tfrac12$
		to obtain for $\testfunTHREE \in H^1$
		\begin{equation}
			\kappa^2 \norm{L^3}{\testfunTHREE}^3 \lesssim 
			\kappa^{1/2} \bigl( \kappa \norm{L^2}{\testfunTHREE} \bigr)^{3/2} \norm{L^6}{\testfunTHREE}^{3/2}
			\lesssim \kappa^{1/2} \Honekappa{\testfunTHREE}^3,
		\end{equation}
		and similarly with $\tfrac{1}{4} = \tfrac{\theta}{2} + \tfrac{1-\theta}{6}$ for $\theta = \tfrac14$
		\begin{equation}
			\kappa^2 \norm{L^4}{\testfunTHREE}^4 \lesssim 
			\kappa \bigl( \kappa \norm{L^2}{\testfunTHREE} \bigr) \norm{L^6}{\testfunTHREE}^3
			\lesssim \kappa \Honekappa{\testfunTHREE}^4,
		\end{equation}
		and the second claim is established.
	\end{proof}
	
	We can finally give the proof of our main result.
	
	\begin{proof}[Proof of Theorem~\ref{thm:main}]
		We mainly collect the results shown in Proposition~\ref{prop:convergence_ahat_general}, Lemma~\ref{lem:convergence_ahat_general_L2}, 
		together with the $L^2$-estimate in
		Lemma~\ref{lem:proj_L2},
		and Lemma~\ref{lem:convergence_energy},
		and the claims are established.
	\end{proof}

	\subsection{Application to Lagrange finite elements}
	
	In this section, we consider the linear Lagrange finite element space $\VSh$ as defined \eqref{eq:defVh}. In order to derive the corresponding error estimates through verifying the assumptions of Theorem~\ref{thm:main}, we require the $L^2$-orthogonal projection onto the ansatz space $\VSh$ as an auxiliary projection. We recall the $L^2$-projection 
	for $\testfunTWO \in L^2$ as
	\begin{equation}
		\ipsymLtwo{\Ltwoproj \testfunTWO}{\testfunh} = \ipsymLtwo{ \testfunTWO}{\testfunh} \quad \mbox{for all} \quad \testfunh \in \VSh.
	\end{equation}
	In the following lemma, we provide corresponding estimates in the $\HonekappaSpace$-norm
	which are the first step towards verifying part \bulletpoint{b} in Assumption~\ref{ass:FEM_space}.

	\begin{lemma} \label{lem:proj}
	
	\begin{itemize}
	\item[{\normalfont (a)}]  The $L^2$-projection $\Ltwoproj$ is stable in 
		$\HonekappaSpace$, i.e., it holds
		\begin{align}
			\Honekappa{\Ltwoproj \testfun} &\lesssim \Honekappa{\testfun},
			\qquad
			\testfun \in H^1,
		\end{align}
		where the constant is independent of $\h$ and $\kappa$.
	\item[{\normalfont (b)}]  
		For all $\testfunFOUR\in H^2$ it holds
		\begin{align}
			\Honekappa{\testfunFOUR - \Ltwoproj \testfunFOUR} 
			&\lesssim \h	\Htwokappa{\testfunFOUR} ,
		\end{align}	
		where the constant is independent of $\h$ and $\kappa$.
	\item[{\normalfont (c)}]  
		If $\Omega$ is convex and $\testfunFOUR\in \Honeperp$ satisfies for $f \in L^2$ the equation
		$\abilmagstabsym{\testfunFOUR}{\testfun} = \ipsymLtwo{f}{\testfun}$ for all $\testfun \in \Honeperp$,
		then 
		\begin{align}
			\Honekappa{\testfunFOUR - \Ltwoproj \testfunFOUR} 
			&\lesssim \h	\norm{L^2}{f}.
		\end{align}
	\end{itemize}	
	\end{lemma}
	\begin{proof}
		Due to \eqref{eq:H1_stab_L2_proj}, standard arguments lead to the bounds on the $L^2$-projection in part \bulletpoint{a} and \bulletpoint{b}.
		Part \bulletpoint{c} is a direct consequence of part \bulletpoint{b} and Lemma~\ref{lem:wepo_abilmagstab}.
	\end{proof}

	In the next lemma, we relate the orthogonal projection, which takes into account the orthogonality to $\ci \sol$ in $\ipsymLtwo{\cdot}{\cdot}$, to the $L^2$-projection.
	
	\begin{lemma} \label{lem:proj_Lagrange}
		For $\kappa \h$ small enough, it  holds the bound
		\begin{equation}
			\Honekappa{ \testfun - \projLtwoisol (\testfun) } 
			\lesssim
			\Honekappa{  \testfun - \Ltwoproj \testfun } ,
			\quad
			\testfun \in \Honeperp.
		\end{equation}
	\end{lemma}
	
	\begin{proof}
		For $\testfunh \in \VSh$ we let $\Ltwotwist: \VSh \rightarrow \VShperp$
		be the mapping that adjusts the angle to $\ci \sol$ via
		\begin{align}
			\Ltwotwist(\testfunh) \coloneqq 
			\testfunh - 
			\frac{  \ipsymLtwo{\testfunh}{\ci \sol}}{
				\ipsymLtwo{\Ltwoproj(\ci \sol)}{\ci \sol} }
			\Ltwoproj(\ci \sol).
		\end{align}
		We this we obtain for any $\testfun \in \Honeperp$
		\begin{align}
			& \, \Honekappa{ \testfun - \projLtwoisol (\testfun) } 
			\\
			&\lesssim \Honekappa{  \testfun - (\Ltwotwist \circ \Ltwoproj  )\testfun } 
			\le  \Honekappa{  \testfun - \Ltwoproj \testfun } 
			+ 
			\frac{\ipsymLtwo{\Ltwoproj \testfun - \testfun}{\ci \sol}}{\ipsymLtwo{\Ltwoproj(\ci \sol)}{\ci \sol}}
			\Honekappa{ \Ltwoproj(\ci \sol) } 
			\\
			&\lesssim \Honekappa{  \testfun - \Ltwoproj \testfun } 
			+ 
			\norm{L^2}{ \testfun - \Ltwoproj \testfun }
			\frac{ \| \sol \|_{L^2} }{
				\ipsymLtwo{\Ltwoproj(\ci \sol) - \ci \sol}{\ci \sol}
				+  \norm{L^2}{  \sol }^2  }  \Honekappa{ \sol } 
			\\
			&\lesssim  \Honekappa{  \testfun - \Ltwoproj \testfun } + \frac{\kappa}{1 -  c \kappa \h} \norm{L^2}{  \testfun - \Ltwoproj \testfun }
			\lesssim \Honekappa{  \testfun - \Ltwoproj \testfun },
		\end{align}
		where we used in the last step that $\norm{L^2}{\Ltwoproj(\ci \sol) - \ci \sol} \lesssim h \norm{H^1}{\sol} \lesssim \kappa h$ holds. 
	\end{proof}
	
	These preparations lead to the error bounds for our first application.
	
	\begin{proof}[Proof of Corollary~\ref{cor:main_Lagrange}]
		From Lemmas~\ref{lem:proj} and \ref{lem:proj_Lagrange}, we obtain that Assumption~\ref{ass:FEM_space} holds,
		and thus we can use the bounds in Theorem~\ref{thm:main}.
		In addition, we recall that $\Omega$ is assumed to be convex, and, hence, the approximation estimates
		due to
		Lemmas~\ref{lem:proj} and \ref{lem:proj_Lagrange} yield
		\begin{equation}
			\Honekappa{\sol - \projLtwoisol \sol}
			\lesssim 
			\h 
			\Htwokappa{\sol}
			\lesssim \kappa^2 \h ,
		\end{equation}
		where we used Theorem~\ref{thm:cont_minimizer} for the last step.
		This establishes the claims.
	\end{proof}

\section{Relaxed $\kappa$-dependencies in LOD spaces}
\label{sec:lod}
In this final section we present a nonstandard application of the abstract approximation result in Theorem~\ref{thm:main}. For that we consider spaces based on the so-called Localized Orthogonal Decomposition (LOD).  LOD spaces were originally developed in the context of elliptic multiscale problems with rough coefficients to efficiently handle low regularity and unresolved scales \cite{MaP14}. An introduction to the methodology is given in the textbook by M{\aa}lqvist and Peterseim \cite{MaP21} and the review article by Altmann et al. \cite{AHP21acta}. Recently, new applications of these spaces emerged in the field of quantum mechanics where they were used to boost the performance of traditional discretizations \cite{HePer21,HeW22,WuZh22}. As we will see see, the Ginzburg-Landau equation could be yet another promising application of LOD spaces in the context of quantum physics.

To define suitable LOD spaces for the GLE and to characterize its approximation properties in an abstract way, we start from a linear Lagrange finite element space $\VSh$ as defined in \eqref{eq:defVh} and  assume that the underlying triangulation $\Th$ is shape-regular and quasi-uniform. The LOD space is now constructed from $\VSh$ by applying the inverse of a differential operator to the functions of $\VSh$. In our case, we use the differential operator associated with the bilinear form $\abilmagstabsym{\cdot,\cdot}$. The construction is made precise in the following definition.
\begin{definition}[LOD spaces]
	\label{definition:LODspace}
	Let $\abilmagstabsym{\cdot}{\cdot}$ denote the symmetric, continuous and coercive bilinear form on $H^1(\Omega,\mathbb{C})$ given by \eqref{eq:def_forms_abil} and
	let $\hat{\mathcal{A}}_{\kappa}^{-1}$ denote the corresponding solution operator on $L^2$, i.e., for $f \in L^2(\Omega,\C)$ the image $\hat{\mathcal{A}}_{\kappa}^{-1} f \in H^1$ is given by the solution to
	\begin{align}
		\abilmagstabsym{ \,\hat{\mathcal{A}}_{\kappa}^{-1}f \, }{ \, \testfun \, } = \ipsymLtwo{f }{ \testfun} \qquad \mbox{ for all  }  \testfun \in H^1.
	\end{align}
	With this definition, the LOD space based on $\abilmagstabsym{\cdot}{\cdot}$ and $\VSh$ is given by
	\begin{align}
		\VShLOD := \hat{\mathcal{A}}_{\kappa}^{-1} \VSh.
	\end{align}
\end{definition}
We note that the above definition of LOD spaces formally differs from the construction given in the classical references \cite{MaP14,HeP13,HeM14}. However, the characterizations are indeed equivalent as can be extracted from e.g., \cite{HaP23} and \cite{AHP21acta}.

From a practical perspective it is also important to note that the space $\VShLOD$ admits a quasi-local basis, i.e., basis functions that are (super-)exponentially decaying in distances of the mesh size $h$. Details on the practical computation/approximation of such basis functions are given in \cite{EHMP19} and recent super-localization strategies are presented in \cite{HaP23}. Corresponding numerical errors that might arise from the approximation of basis functions are well understood \cite{AHP21acta} and will be for brevity disregarded in the following error analysis.

The approximation properties of the idealized space $\VShLOD$ are summarized in the following proposition.
\begin{proposition}[Approximation properties of $\VShLOD$]
	\label{proposition:LOD-abstract-est}
	Let $\VShLOD$ be the LOD-space from Definition \ref{definition:LODspace} and let $f \in L^2$ be given. 
	If $\sol \in H^1$ denotes the solution to
	$$
	\abilmagstabsym{\sol}{\testfun} = \ipsymLtwo{f }{ \testfun} \qquad \mbox{ for all  }  \testfun \in H^1
	$$
	and if $\RitzprojLOD \sol \in \VShLOD$ denotes the corresponding $\abilmagstabsym{\cdot}{\cdot}$-Ritz-projection of $\sol$ in $\VShLOD$, then it holds
	\begin{align}
		\label{eqn:abstract-estimate-LOD}
		\Honekappa{ \sol - \RitzprojLOD \sol} \, \lesssim\, h \, \norm{L^2}{f - \Ltwoproj f},
	\end{align}
	where we recall $\Ltwoproj : L^2 \rightarrow \VSh$ as the $L^2$-projection on $\VSh$.
	The hidden constant in \eqref{eqn:abstract-estimate-LOD} is generic and depends on the coercivity and continuity constants of $\abilmagstabsym{\cdot}{\cdot}$, as well as the mesh regularity, but it does not depend on $h$ and $\kappa$.
	
	Furthermore, for every $\phi \in H^1$ there exists a unique decomposition such that
	\begin{align}
		\label{equation:LODdecomposition}
		\phi = \phi^{\LOD} + \phi_0,
		 \quad \mbox{where } \,\,
		 \phi^{\LOD}\in\VShLOD,\quad \Ltwoproj \phi_0 = 0
		\quad 
		\mbox{ and } 
		\quad 
		\abilmagstabsym{\phi^{\LOD}}{\phi_0} = 0.
	\end{align}
\end{proposition}
The result is standard and can be for instance found in \cite{HeW22} for homogeneous Dirichlet boundary conditions. For generalizations to higher order FE spaces and to only piecewise smooth source terms $f$, we refer to \cite{Mai21}.

To apply the general error estimates in Theorem \ref{thm:main}, we need to verify Assumption \ref{ass:FEM_space} for the LOD space $\VShLOD$.
Analogously, to standard Lagrange finite elements, it is also possible to quantify the approximation properties of $\RitzprojLOD $ for general smooth functions. This is done in the following lemma.
\begin{lemma}
	\label{lemma:RitzLOD-properties}
	Let $\VShLOD$ be the LOD-space from Definition \ref{definition:LODspace} and let the corresponding Ritz-projection w.r.t. $\abilmagstabsym{\cdot}{\cdot}$ be denoted by $\RitzprojLOD: H^1 \rightarrow \VShLOD$.
	Then, for every $\testfunTHREE \in H^2$ with $\nabla \testfunTHREE \cdot \nu =0$, there is a $f_\testfunTHREE \in L^2$ such that
	\begin{equation}
		\abilmagstabsym{\testfunTHREE}{\testfun} = \ipsymLtwo{f_\testfunTHREE}{\testfun} \quad \mbox{and} \quad \norm{L^2}{f_\testfunTHREE} \lesssim \Htwokappa{\testfunTHREE}.
	\end{equation}
	Consequently, for all $\testfunTHREE \in H^2$  with $\nabla \testfunTHREE \cdot \nu =0$ it holds
	\begin{align}
		\Honekappa{\testfunTHREE - \RitzprojLOD \testfunTHREE}  \lesssim \h \Htwokappa{\testfunTHREE}.
	\end{align}
\end{lemma}

\begin{proof}
	From \eqref{eq:expansion_ahat}, we obtain using integration by parts
	\begin{equation}
		\abilmagstabsym{\testfunTHREE}{\testfun} 
		=  	  
		\Real \int_{\Omega}
		\bigl( - \Delta \testfunTHREE + 
		\stabPar^2 \testfunTHREE  
		+ 2 \ci \kappa \MagF \nabla \testfunTHREE 
		+ \kappa^2 |\MagF|^2 \testfunTHREE  \bigr) \testfun^*
		\,dx  \eqqcolon \ipsymLtwo{f_\testfunTHREE}{\testfun}
	\end{equation}
	and the bound for $\norm{L^2}{f_\testfunTHREE}$ follows. Proposition \ref{proposition:LOD-abstract-est} finishes the second part of the lemma.
\end{proof}

As the set of $H^2$-functions with a vanishing normal derivative on $\partial \Omega$ (for a polygonal Lipschitz domain) is dense in $H^1$ (cf. \cite{Droniou}), this lemma establishes property \bulletpoint{a} in Assumption~\ref{ass:FEM_space}.
For the second property, we need a variant of this result given in the next two lemmas. First, we give an estimate on the (standard) Ritz projection.

\begin{lemma}
	\label{lemma:RitzLOD-properties_v2}
	Let $\VShLOD$ be the LOD-space from Definition~$\ref{definition:LODspace}$ with corresponding Ritz-projection w.r.t. $\abilmagstabsym{\cdot}{\cdot}$ given by $\RitzprojLOD \colon H^1 \rightarrow \VShLOD$.
	For $f \in L^2$ let $\testfunTHREE \in \Honeperp$ be the solution of
	\begin{equation}
		\abilmagstabsym{\testfunTHREE}{\testfun} = \ipsymLtwo{f}{\testfun} \quad \mbox{for all } \quad \testfun \in \Honeperp.
	\end{equation}
	Then, it holds
	\begin{align}
		\Honekappa{\testfunTHREE - \RitzprojLOD \testfunTHREE}  \lesssim \h \norm{L^2}{f}.
	\end{align}
\end{lemma}

\begin{proof}
	As in the proof of Lemma~\ref{lem:wepo_abilmagstab} in \eqref{eq:var_form_extended_test_fun}, we know that $\testfunTHREE$ solves the variational problem also tested against all $\testfun \in H^1$ for some modification of $f$ which is bounded in $L^2$ by $\norm{L^2}{f}$. Hence, the assertion follows from
	Proposition~\ref{proposition:LOD-abstract-est}.
\end{proof}

However, for property (b) in Assumption~\ref{ass:FEM_space}, we 
further need a bound on the Ritz projection which preserves the orthogonality with respect to $\ci \sol$. 
The following lemma shows that one can reduce these error bounds
to those established in Lemma~\ref{lemma:RitzLOD-properties_v2}.

\begin{lemma}
	\label{lemma:ritzprojorth-LOD}
	Let again $\RitzprojLOD: H^1 \rightarrow \VShLOD$ denote the Ritz-projection onto the LOD-space $\VShLOD$ and let
	\begin{align}
		\projLtwoisolLOD : \Honeperp \rightarrow \VShLOD \cap (\ci \sol)^\perp
	\end{align}
	denote the corresponding Ritz-projection onto $\VShLOD \cap (\ci \sol)^\perp$. 
	If $h$ is small enough, in particular $h \lesssim \kappa^{-1}$, then it holds for all $\testfun \in \Honeperp$
	\begin{align}
		\Honekappa{\testfun - \projLtwoisolLOD \testfun}  \lesssim \Honekappa{\testfun - \RitzprojLOD \testfun}. 
	\end{align}
\end{lemma}

\begin{proof}
	To proceed as in the proof of Lemma \ref{lem:proj_Lagrange}, we note
	that by the LOD-decomposition \eqref{equation:LODdecomposition} we have $\Ltwoproj\left( \ci \sol - \RitzprojLOD (\ci \sol) \right) = 0$. Hence, with the approximation properties of $\Ltwoproj$:
	\begin{align}
		\label{equation:L2-est-RitzprojLOD}
		\norm{L^2}{ \RitzprojLOD (\ci \sol)  - \ci \sol }
		\lesssim h \, \Honekappa{ \RitzprojLOD (\ci \sol)  - \ci \sol }
		\lesssim h \, \Honekappa{ \ci \sol } \lesssim h \, \kappa.
	\end{align}
	This implies for all $\testfun \in \Honeperp$ 
	\begin{align}
		& \,\Honekappa{ \testfun - \projLtwoisolLOD \testfun } 
		\\
		&\lesssim \Honekappa{  \testfun - \Bigl( \RitzprojLOD \testfun - 
			\frac{  \ipsymLtwo{\RitzprojLOD \testfun}{\ci \sol}}{\ipsymLtwo{\RitzprojLOD(\ci \sol)}{\ci \sol} } \RitzprojLOD(\ci \sol) \Bigr) } \\
		&\le  \Honekappa{  \testfun - \RitzprojLOD \testfun } 
		+ 
		\frac{\ipsymLtwo{\RitzprojLOD \testfun - \testfun}{\ci \sol}}{\ipsymLtwo{\RitzprojLOD(\ci \sol)}{\ci \sol}}
		\Honekappa{ \RitzprojLOD(\ci \sol) } 
		\\
		&\hspace{-4pt}\overset{\eqref{equation:L2-est-RitzprojLOD}}{\lesssim} \Honekappa{  \testfun - \RitzprojLOD \testfun } 
		+ 
		\norm{L^2}{ \testfun - \RitzprojLOD \testfun }
		\frac{ \| \sol \|_{L^2} }{
			\ipsymLtwo{\RitzprojLOD(\ci \sol) - \ci \sol}{\ci \sol}
			+  \norm{L^2}{  \sol }^2  }  \Honekappa{ \sol } 
		\\
		&\lesssim  \Honekappa{  \testfun - \RitzprojLOD \testfun } + \frac{\kappa}{1 -  c \kappa \h} \norm{L^2}{  \testfun - \RitzprojLOD \testfun }
		\lesssim \Honekappa{  \testfun - \RitzprojLOD \testfun }. 
	\end{align}
\end{proof}
Lemma \ref{lemma:ritzprojorth-LOD} together with Theorem \ref{thm:main} guarantees that the $\HonekappaSpace$-error between an exact solution $u$ and a corresponding approximation in the LOD-space is bounded by the projection error $\Honekappa{ \sol - \RitzprojLOD \sol}$. The next lemma quantifies this error.

\begin{lemma}
	\label{lemma:estimate-u-LODspace}
	Let $\sol$ be a minimizer of \eqref{eq:energy_functional_times_kappa2} and let $\RitzprojLOD: H^1 \rightarrow \VShLOD$ be the Ritz-projection onto $\VShLOD$. Then it holds at least
	\begin{align}
		\Honekappa{ \sol - \RitzprojLOD \sol}  \lesssim \kappa^{3}\, \h^{2} 
	\end{align}
	and, if $\Omega$ is convex, we have $u\in H^2$ and the estimate improves to
	\begin{align}
		\Honekappa{ \sol - \RitzprojLOD \sol}  \lesssim \kappa^{4}\, \h^{3}. 
	\end{align}
\end{lemma}

\begin{proof}
	We want to apply Proposition \ref{proposition:LOD-abstract-est}. By $\energy'(\sol) = 0 $ we have for every $\phi \in \HonekappaSpace$ that
	\begin{align}
		\abilmagstabsym{\sol}{ \testfun}  &= \stabPar^2 
		\Real  \int_\Omega \sol \testfun^* 
		\,dx 
		-
		\kappa^2 	\Real  \int_\Omega \bigl( |\sol|^2 -1 \bigr)  \sol \testfun^* 
		\,dx 
		= \ipsymLtwo{\stabPar^2 \sol - \kappa^2 \bigl( |\sol|^2 -1 \bigr)  \sol  }{\testfun} .
	\end{align}
	Since $\stabPar^2 \sol - \kappa^2 \bigl( |\sol|^2 -1 \bigr)  \sol$ is at least in $H^1$ and even in $H^2$ for convex domains, one easily verifies that for $s=0,1$
	\begin{equation}
		\norm{H^s}{\stabPar^2 \sol - \kappa^2 \bigl( |\sol|^2 -1 \bigr)  \sol }
		\lesssim
		\kappa^2 \norm{H^s}{\sol  } 
		\leq \kappa^{s+2},
	\end{equation}
	and for $s=2$
	\begin{equation}
		\norm{H^2}{\stabPar^2 \sol - \kappa^2 \bigl( |\sol|^2 -1 \bigr)  \sol }
		\lesssim
		\kappa^2 	\bigl( 	\norm{W^{1,4}}{\sol  }^2  + 	\norm{H^2}{\sol  } \bigr) 
		\leq \kappa^{4},
	\end{equation}
	where we used the bounds from Theorem \ref{thm:cont_minimizer} and in particular repeatedly $|\sol| \leq 1$. The estimate now follows with Proposition \ref{proposition:LOD-abstract-est} and standard estimates for the $L^2$-projection $\Ltwoproj$ on $P1$ finite element spaces.
\end{proof}
By collecting the previous results we obtain our final main result which shows the superapproximation properties of the LOD space, even on nonconvex domains.
\begin{theorem} \label{thm:final_bounds_LOD_H1}
	Let Assumption~\ref{ass:cinfsup} hold and let $\h$ be sufficiently small in the sense of Theorem \ref{thm:main}. If $\VShLOD$ denotes the LOD-space from Definition \ref{definition:LODspace} and if $\solhLOD \in \VShLOD$ is a corresponding minimizer of the Ginzburg--Landau energy with
	\begin{equation} 
		\energy(\solhLOD) = \inf\limits_{\testfun \in \VShLOD } E(\testfun),
	\end{equation}
	then, there is neighborhood $\Nbh \subset H^1(\Omega)$ of $\solhLOD$ and a unique minimizer $\sol \in \Nbh$ of \eqref{eq:energy_functional_times_kappa2} with $\ipsymLtwo{\solhLOD}{\ci \sol} = 0$ and such that
	\begin{align}
		\Csolinv \, \norm{L^2(\Omega)}{ \sol - \solhLOD}  + \h \, \Honekappa{ \sol - \solhLOD} &\lesssim \kappa^{3} \h^{3}   ,
	\end{align}
	and for convex domains $\Omega$ (and consequently $H^2$-solutions) it even holds
	\begin{align}
		\Csolinv \, \norm{L^2(\Omega)}{ \sol - \solhLOD}  + \h \, \Honekappa{ \sol - \solhLOD} &\lesssim \kappa^{4} \h^{4} .
	\end{align}
	\end{theorem}
	
	\begin{proof}
		Proposition \ref{proposition:LOD-abstract-est} and Lemmas~\ref{lemma:RitzLOD-properties}, \ref{lemma:RitzLOD-properties_v2}, and \ref{lemma:ritzprojorth-LOD} guarantee that Assumption \ref{ass:FEM_space} is fulfilled for $\VShLOD$. Hence, we can apply Theorem \ref{thm:main} together with  Lemmas~\ref{lemma:ritzprojorth-LOD} and~\ref{lemma:estimate-u-LODspace} to conclude that for all sufficiently small $h$ and for $u \in H^s$ with $s\in \{1,2\}$ it holds
		\begin{align} 
			\Honekappa{ \sol - \solhLOD} \lesssim \Honekappa{\sol - \projLtwoisolLOD \sol}   \lesssim   \Honekappa{ \sol - \RitzprojLOD \sol}  \lesssim \kappa^{s+2}\, \h^{s+1} ,
		\end{align}
		and
		\begin{align}
			\norm{L^2}{ \sol - \solhLOD} \lesssim \Csol \h\, \Honekappa{ \sol - \solhLOD} \lesssim \Csol\, \kappa^{s+2}\, \h^{s+2}.
		\end{align}
	\end{proof}

	\begin{remark}
		It is worth to note that, in LOD-spaces, one can also improve the smallness condition on $\kappa \Csol \h$ required for the inf-sup condition in  Lemma~\ref{lem:du_F_inf_sup_solvability_discrete}. 
		In fact, a precise inspection of the proof leads to a smallness condition on $\kappa^2 \Csol \h^2$, which is in general weaker if $1 \lesssim \Csol$. However, since the abstract result contains a term of the form $1 + \kappa \Csol \h$, one cannot exploit this any further in the error analysis, and we thus refrain from giving the proof here.
	\end{remark}

\section*{Acknowledgments}
The authors would like to thank Qiang Du for the very helpful comments on the introduction of the paper and for pointing us to additional important literature references.

\appendix 

\section{Matrix representation of $E^{\prime\prime}(u_h)$}
\label{appendix}

In the following we briefly discuss the matrix representation of $E^{\prime\prime}(u_h)$ and the computation of the lower-most eigenvalues. 
For that, let $\mathcal{N}_h = \{ z_1, \cdots, z_N \}$ denote the set of nodes of the mesh $\mathcal{T}_h$ and let $\phi_j \in \VSh$ denote the {\it real} nodal shape function with the property
$$
\phi_j(z_k) = \delta_{jk}  \qquad \mbox{for all } 1 \le j,k \le N.
$$
Consequently, a nodal basis of $\VSh$ is given by the set
\begin{eqnarray*}
	\{ \phi_1, \ci  \phi_1 , \phi_2, \ci  \phi_2 ,  \cdots, \phi_N, \ci \phi_N \}.  
\end{eqnarray*}
We are seeking of a matrix representation of the operator $\langle \energy''(\solh) v_h , w_h \rangle $ for arbitrary $v_h,w_h \in \VSh$. By expanding $v_h$ and $w_h$ in terms of the nodal basis functions we obtain
\begin{align}
	v_h = \sum_{j=1}^N v_j^{\mbox{\tiny\rm Re}} \, \phi_j +  \sum_{j=1}^N v_j^{\mbox{\tiny\rm Im}} \, \ci \, \phi_j  
	\qquad
	\mbox{and}
	\qquad
	w_h = \sum_{j=1}^N w_j^{\mbox{\tiny\rm Re}} \, \phi_j +  \sum_{j=1}^N w_j^{\mbox{\tiny\rm Im}} \, \ci \, \phi_j  
\end{align}
for corresponding coefficient vectors $\mathbf{v}^{\mbox{\tiny\rm Re}},\mathbf{v}^{\mbox{\tiny\rm Im}},\mathbf{w}^{\mbox{\tiny\rm Re}},\mathbf{w}^{\mbox{\tiny\rm Im}}\in \mathbb{R}^N$ with
\begin{eqnarray*}
	\mathbf{v}^{\mbox{\tiny\rm Re}} =  
	\left(\begin{matrix}
		v_1^{\mbox{\tiny\rm Re}} \\
		\vdots \\
		v_N^{\mbox{\tiny\rm Re}}
	\end{matrix}\right) 
	,\quad
	\mathbf{v}^{\mbox{\tiny\rm Im}} =  
	\left(\begin{matrix}
		v_1^{\mbox{\tiny\rm Im}} \\
		\vdots \\
		v_N^{\mbox{\tiny\rm Im}}
	\end{matrix}\right), 
	\quad
	\mathbf{w}^{\mbox{\tiny\rm Re}} =  
	\left(\begin{matrix}
		w_1^{\mbox{\tiny\rm Re}} \\
		\vdots \\
		w_N^{\mbox{\tiny\rm Re}}
	\end{matrix}\right),
	\quad
	\mathbf{w}^{\mbox{\tiny\rm Im}} =  
	\left(\begin{matrix}
		w_1^{\mbox{\tiny\rm Im}} \\
		\vdots \\
		w_N^{\mbox{\tiny\rm Im}}
	\end{matrix}\right). 
\end{eqnarray*}
Hence, we can write
\begin{align}
	\langle \energy''(\solh) v_h , w_h \rangle  = 
	\left(\begin{matrix}
		\mathbf{w}^{\mbox{\tiny\rm Re}} \\
		\mathbf{w}^{\mbox{\tiny\rm Im}}
	\end{matrix}
	\right)^{\top}	
	\left( \begin{matrix}
		\mathbf{A}(u_h)^{\mbox{\tiny\rm RR}} & \mathbf{A}(u_h)^{\mbox{\tiny\rm IR}} \\
		\mathbf{A}(u_h)^{\mbox{\tiny\rm RI}} &\mathbf{A}(u_h)^{\mbox{\tiny\rm II}}
	\end{matrix} \right)
	\left(\begin{matrix}
		\mathbf{v}^{\mbox{\tiny\rm Re}} \\
		\mathbf{v}^{\mbox{\tiny\rm Im}}
	\end{matrix}
	\right),
\end{align}
for a block matrix $\mathbf{A}(u_h) \in \R^{2N \times 2N}$. Recalling the representation of $\energy''(\solh)$ as
\begin{align}
	\dualp{\energy''(\solh) v_h}{w_h}	
	&=
	\Real  \int_\Omega \bigl( \nabla v_h + \ci \kappa \MagF w_h \bigr)  \cdot \bigl( \nabla v_h + \ci \kappa \MagF w_h \bigr)^*   
	\\
	&\quad
	+
	\kappa^2 
	\bigl(  ( |\solh|^2 -1  ) v_h w_h^*  + \solh^2 v_h^* w_h^* + |\solh|^2 v_h w_h^* \bigr)
	\,dx,
\end{align}
we compute the corresponding blocks of  $\mathbf{A}(u_h) $ as follows. For the upper left block we obtain straightforwardly
\begin{eqnarray*}
	\mathbf{A}(u_h)^{\mbox{\tiny\rm RR}}_{jk}
	&=& \Real \int_\Omega \nabla \phi_k \cdot \nabla \phi_j + \kappa^2 \left( |\MagF|^2 +  2 |\solh|^2 -1 + \solh^2 \right)\phi_k \phi_j  \,dx \\
	&=&  \int_\Omega \nabla \phi_k \cdot \nabla \phi_j + \kappa^2 \left( |\MagF|^2 -1 + 3 \Real(u_h)^2 + \Imag(u_h)^2 \right)\phi_k \phi_j  \,dx.
\end{eqnarray*}
For the upper right block we have 
\begin{align}
	&\, \mathbf{A}(u_h)^{\mbox{\tiny\rm IR}}_{jk}
	\\
	&= \Real  \int_\Omega \ci \bigl( \nabla \phi_k + \kappa \MagF \ci \phi_k \bigr)  \cdot \bigl( \nabla \phi_j + \ci \kappa \MagF \phi_j \bigr)^*   
	+
	\ci \kappa^2 
	\bigl(  ( |\solh|^2 -1  )  - \solh^2 + |\solh|^2 \bigr) \phi_k \phi_j
	\,dx \\
	&=  \int_\Omega \kappa \left( \phi_j  \MagF \cdot \nabla \phi_k - \phi_k \MagF \cdot  \nabla \phi_j \right)
	-
	\kappa^2 \Real( \ci  
	\solh^2 ) \phi_k \phi_j
	\,dx \\
	&=  \int_\Omega \kappa \left( \phi_j  \MagF \cdot \nabla \phi_k - \phi_k \MagF \cdot  \nabla \phi_j \right)
	+ 2
	\kappa^2 \Real(u_h) \Imag(u_h) \, \phi_k \phi_j
	\,dx . 
\end{align}
%
Analogously, we obtain for the upper left block 
\begin{eqnarray*}
	\mathbf{A}(u_h)^{\mbox{\tiny\rm RI}}_{jk}
	&=& \int_\Omega \kappa ( \phi_k \MagF \cdot \nabla \phi_j  - \phi_j  \MagF \cdot \nabla \phi_k )
	+ 2 \kappa^2 \Real(u_h) \Imag(u_h) \, \phi_k \phi_j \,dx.
\end{eqnarray*}
Finally, the lower left block is given by
\begin{eqnarray*}
	\mathbf{A}(u_h)^{\mbox{\tiny\rm II}}_{jk}
	&=&  \int_\Omega  \nabla \phi_k \cdot  \nabla \phi_j + \kappa^2 \left(  |\MagF|^2 
	-1 + \Real(u_h)^2 + 3 \Imag(u_h)^2)\right) \phi_k \phi_j \,dx.
\end{eqnarray*}
By assembling the blocks in a standard way, we obtain the matrix $\mathbf{A}(u_h)$. Since $\langle E^{\prime\prime}(u_h)\,\cdot,\cdot\rangle$ is symmetric, $\mathbf{A}(u_h)$ must be symmetric too. This is also easily seen by the observations that
\begin{align}
	\mathbf{A}(u_h)^{\mbox{\tiny\rm RR}}_{jk}=\mathbf{A}(u_h)^{\mbox{\tiny\rm RR}}_{kj}, \qquad
	\mathbf{A}(u_h)^{\mbox{\tiny\rm II}}_{jk}=\mathbf{A}(u_h)^{\mbox{\tiny\rm II}}_{kj},
	\qquad 
	\mbox{and}
	\quad \mathbf{A}(u_h)^{\mbox{\tiny\rm IR}}_{kj} = \mathbf{A}(u_h)^{\mbox{\tiny\rm RI}}_{jk}.
\end{align}
Note that a matrix representation of $m(v_h,w_h)$ is straightforwardly given by
\begin{align}
	m(v_h,w_h) = 
	\left(\begin{matrix}
		\mathbf{w}^{\mbox{\tiny\rm Re}} \\
		\mathbf{w}^{\mbox{\tiny\rm Im}}
	\end{matrix}
	\right)^{\top}\left( \begin{matrix}
		\mathbf{M} & 0 \\
		0 & \mathbf{M}
	\end{matrix} \right)
	\left(\begin{matrix}
		\mathbf{v}^{\mbox{\tiny\rm Re}} \\
		\mathbf{v}^{\mbox{\tiny\rm Im}}
	\end{matrix}
	\right),
\end{align}
where $\mathbf{M} \in \mathbb{R}^{N \times N}$ is the conventional mass matrix with entries $\mathbf{M}_{jk}= \int_\Omega \phi_k \phi_j \,dx$. Hence, the eigenvalues of $ E^{\prime\prime}(u_h)$ in the sense of definition \eqref{eigenvalues-secvarE} are given by the system 
\begin{align}
	\left( \begin{matrix}
		\mathbf{A}(u_h)^{\mbox{\tiny\rm RR}} & \mathbf{A}(u_h)^{\mbox{\tiny\rm IR}} \\
		\mathbf{A}(u_h)^{\mbox{\tiny\rm RI}} &\mathbf{A}(u_h)^{\mbox{\tiny\rm II}}
	\end{matrix} \right)
	\left(\begin{matrix}
		\mathbf{v}^{\mbox{\tiny\rm Re}}_{\ell} \\
		\mathbf{v}^{\mbox{\tiny\rm Im}}_{\ell} 
		\end{matrix}
		\right)
		=
		\lambda_{\ell} 
		\left( \begin{matrix}
			\mathbf{M} & 0 \\
			0 & \mathbf{M}
		\end{matrix} \right)
		\left(\begin{matrix}
			\mathbf{v}^{\mbox{\tiny\rm Re}}_{\ell} \\
			\mathbf{v}^{\mbox{\tiny\rm Im}}_{\ell}
		\end{matrix}
		\right).
	\end{align}
	Due to the symmetry of the matrices, the smallest eigenvalues can be easily computed with a standard method such as the inverse power iteration.

\end{document}